\documentclass[12pt]{article}

\usepackage{multirow}
\usepackage{array,booktabs}

\usepackage[ruled]{algorithm2e}

\usepackage{authblk}
\usepackage{amsthm}
\newtheorem{theorem}{Theorem}
\newtheorem{lemma}{Lemma}
\newtheorem{corollary}{Corollary}
\newtheorem{definition}{Definition}
\newtheorem{remark}{Remark}
\usepackage[toc,page,title,titletoc,header]{appendix}

\usepackage{times}
\usepackage{bm}
\usepackage{natbib}
\usepackage{color}

\def\T{{ \mathrm{\scriptscriptstyle T} }}

\def\sqrt{\surd}

\usepackage{amsmath,amssymb}
\usepackage{caption}
\usepackage{graphicx}
\usepackage{times}
\usepackage{subfigure}
\usepackage{epstopdf}

\newcommand{\bsb}{{}}

\newcommand{\rd}{\,\mathrm{d}}

\newcommand{\bsbX}{{{X}}}
\newcommand{\bsbx}{{{x}}}
\newcommand{\bsby}{{{y}}}
\newcommand{\bsbY}{{{Y}}}

\newcommand{\be}{{{e}}}
\newcommand{\bsbb}{{{\beta}}}

\newcommand{\bsbGamma}{{{\Gamma}}}

\newcommand{\bsbI}{{{I}}}

\newcommand{\bsbZ}{{{Z}}}
\newcommand{\bsbSig}{{{\Sigma}}}

\newcommand{\bsbD}{{{D}}}

\newcommand{\bsbU}{{{U}}}
\newcommand{\bsbV}{{{V}}}

\newcommand{\bsba}{{{\alpha}}}
\newcommand{\bsbA}{{{A}}}

\newcommand{\bsbC}{{{C}}}

\newcommand{\bsbE}{{{\mathcal E}}}

\newcommand{\bsbdelta}{{{\delta}}}

\newcommand{\bsbS}{{{S}}}

\newcommand{\bsbc}{{{c}}}
\newcommand{\bsbB}{{{B}}}

\newcommand{\bsbDelta}{{{\Delta}}}

\newcommand{\outl}{\varrho}

\newcommand{\Proj}{{{\mathcal P}}}
\newcommand{\essinf}{{\mathrm{ess}\inf}}
\newcommand{\EP}{\,{\mathrm{pr}}}
\newcommand{\EE}{\,{E}}

\DeclareMathOperator{\vect}{\mbox{vec}\,}

\def\T{{ \mathrm{\scriptscriptstyle T} }}

\def\^T{{\T}}

\newcommand{\tr}{\mbox{tr}}

\newcommand{\tF}{{\textrm{F}}}

\usepackage{setspace}
\usepackage[normalem]{ulem}

\begin{document}

\title{Robust reduced-rank regression}
%
\author[1]{Yiyuan She}
\author[2]{Kun Chen}
\affil[1]{Department of Statistics, Florida State University}
\affil[2]{Department of Statistics, University of Connecticut}
\date{}
\maketitle

\begin{abstract}
In high-dimensional multivariate regression problems, enforcing low rank in the coefficient matrix offers    effective  dimension reduction, which greatly facilitates parameter estimation and model interpretation. However, commonly-used reduced-rank methods are sensitive to   data corruption, as the low-rank dependence structure between response variables and predictors is easily distorted by outliers. We propose a robust reduced-rank regression approach for joint   modeling and outlier detection. The problem is formulated as a regularized multivariate regression with a sparse mean-shift parametrization, which generalizes and unifies some popular robust multivariate methods. An efficient thresholding-based iterative procedure is developed for optimization. We show that the algorithm is guaranteed to converge, and the  coordinatewise minimum point produced is statistically accurate    under  regularity conditions. Our theoretical investigations focus on nonasymptotic robust analysis, which demonstrates that  joint rank reduction and outlier detection  leads to improved prediction accuracy. In particular,    we show that redescending   $\psi$-functions   can essentially attain the minimax optimal error rate, and    in some less challenging problems   convex regularization guarantees  the same low error rate. The performance of the proposed method is examined  by simulation studies and   real data examples.
\end{abstract}



\section{Introduction}
\label{sec:intro}

Given $n$ observations of $m$ response variables and $p$ predictors, denoted by $\bsby_i\in \mathbb{R}^m$ and $\bsbx_i\in \mathbb{R}^p$ for $i=1,\ldots,n$, we consider the   multivariate regression model
\begin{equation}
\bsbY = \bsbX\bsbB^* +\bsbE, \label{model1}
\end{equation}
where $\bsbY = (\bsby_{1},\ldots,\bsby_n)^\T$, $\bsbX=(\bsbx_1,\ldots,\bsbx_n)^\T$, $\bsbB^*\in \mathbb{R}^{p\times m}$ is an unknown coefficient matrix, and $\bsbE=(\be_1,\ldots,\be_n)^\T \in\mathbb{R}^{n\times m}$ is a random error matrix. Such a high-dimensional multivariate problem,  in which both $p$ and $m$ may be comparable to or even exceed the sample size $n$, has drawn increasing attention in both applied and theoretical statistics.


Conventional least squares linear regression ignores the  multivariate nature of the problem and may  fail when $p$   is large relative to $n$. Dimension reduction holds the key to characterizing the dependence between   responses  and   predictors in a parsimonious way. Reduced-rank regression \citep{anderson1951,izenman1975} achieves this  by restricting the rank of the coefficient matrix, i.e.,
by solving the problem\begin{align}
\min_{\bsbB\in \mathbb R^{p\times m}} \tr\{(\bsbY - \bsbX \bsbB) \bsbGamma (\bsbY - \bsbX \bsbB)^\T\} \quad \mbox{subject to} \ r(\bsbB) \leq r, \label{rrr}
\end{align}
where $\tr(\cdot)$ and $r(\cdot)$ denote   trace and rank, and $\bsbGamma$ is a pre-specified positive definite weighting matrix \citep{reinsel1998}. The ranks   are typically much smaller than   $m$ and $p$. A global solution to \eqref{rrr} can be obtained explicitly. 
See \cite{reinsel1998} for a  comprehensive account of   reduced-rank regression under the classical large-$n$ asymptotic regime. Finite-sample theories  on  rank selection and estimation   accuracy of   the penalized form of reduced-rank regression were developed by \citet{bunea2011}.
The nuclear norm and Schatten $p$-norms can also  be used to  promote  sparsity of the singular values of  $\bsbB$ or $\bsbX\bsbB$; see   \cite{yuan2007}, \cite{koltch2011}, \cite{Rohde2011}, \citet{agarwal2012}, \citet{Foygel2012}, \cite{chen2012ann},  among others. Reduced-rank regression is closely connected with principal component analysis, canonical correlation analysis, partial least squares,  matrix completion, and many other multivariate methods \citep{izenman2008}. 




Although reduced-rank regression can substantially reduce the number of free parameters in multivariate problems,  it is extremely sensitive to outliers, which are bound to occur, and thus in real-world data analysis,   the low-rank structure could easily be masked or distorted. This is even more  serious in high-dimensional or big-data applications. For example, in cancer genetics, multivariate regression is commonly used to explore the associations between genotypical and phenotypical characteristics \citep{vounou2010}, where employing rank regularization can help to reveal  latent regulatory pathways linking the two sets of variables. But   pathway recovery  should  not  be   distorted by  abnormal samples or subjects. As another example,   financial time series, even after stationarity transformation, often contain  anomalies or demonstrate heavier tails than those of a normal distribution, which may jeopardize the recovery of common market behaviors and   asset return forecasting. 


Consider the 52 weekly stock log-return data for nine of the ten largest American corporations in 2004 available from the R package \textsc{mrce} \citep{Rothman2010}, with $\bsby_t\in\mathbb{R}^9$ ($t=1,\ldots,T$) and $T=52$.
Chevron was excluded   due to its drastic  changes  \citep{yuan2007}. The nine time series are shown in Figure \ref{fig:stock}.  For the purpose of  constructing market factors that drive general stock movements, a reduced-rank vector autoregressive model can be used, i.e., $\bsby_{t} =\bsbB^*\bsby_{t-1}+\bsb{e}_t$, with $\bsbB^*$ of low rank. By conditioning on the initial state $\bsby_0$ and assuming the normality of $\bsb{e}_t$, the conditional likelihood leads to a least squares criterion, so the estimation of $\bsbB^*$ can be formulated as a reduced-rank regression problem \citep{Reinsel1997,Ltkepohl2007}. However, as shown in the figure, several stock returns experienced  short-term changes, and the autoregressive structure  makes any outlier in the time series also a leverage point in the covariates.

Using the weekly log-returns in the first 26 weeks for training and those in the last 26 weeks for forecast, we analyzed the data with the reduced-rank regression and the proposed robust reduced-rank regression approach. While both methods resulted in unit-rank models, the robust reduced-rank regression automatically detected three outliers, i.e., the log-returns of Ford at weeks 5 and 17 and the log-return of General Motors at week 5.
These  correspond to two real major market disturbances attributed to the auto industry.   Our robust method automatically took   the   outlying samples into account and led to a more reliable model. 
 Table \ref{table:stock} displays the factor coefficients indicating how the stock returns are related to the estimated factors, and the $p$-values for testing the associations between the estimated factors and the individual stock return series using the data in the last 26 weeks. The stock factor estimated robustly has positive influence over all nine companies, and overall, it   correlates with the   series better according to the reported  $p$-values. The   out-of-sample prediction errors for least squares, reduced-rank regression and robust reduced-rank regression   are $9.97$, $8.85$ and $6.72$, respectively, when measured by  mean square  error, and are   $5.44$, $4.52$ and $3.58$, respectively, when measured by   40\% trimmed mean square error. The robustification of rank reduction resulted in about 20\% improvement in prediction.

\begin{figure}[h]
\begin{center}
{\includegraphics[width=\textwidth]{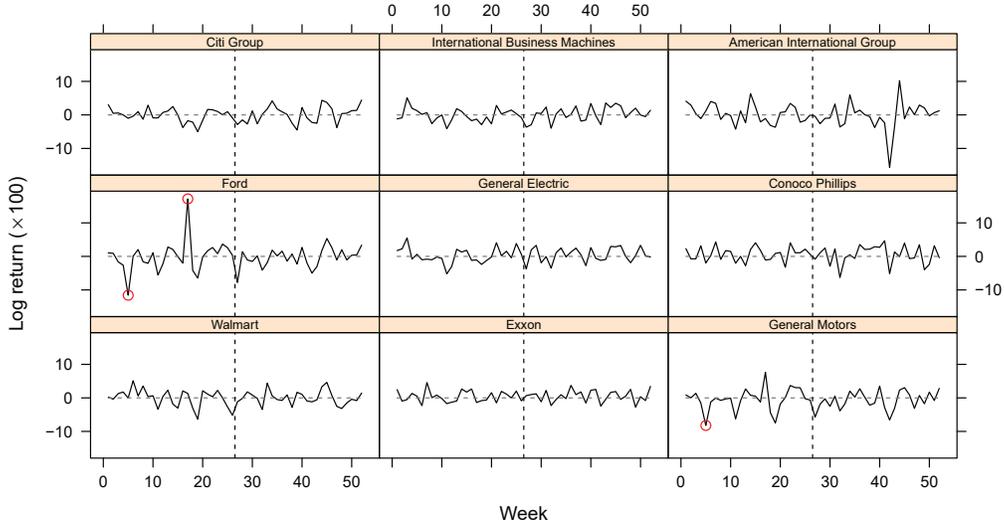}}
\caption{Stock return example: scaled weekly log-returns of stocks in 2004. The log-returns of Ford at weeks 5 and 17 and the log-return of General Motors at week 5 are captured as outliers by fitting robust reduced-rank regression with data in the first 26 weeks; the corresponding points are indicated by the circles. The dashed line in each panel separates the series to two parts, i.e., the first 26 weeks for training and the last 26 weeks for testing. The horizontal line in each panel is drawn at zero height.}\label{fig:stock}
\end{center}
\end{figure}


\begin{table}
\footnotesize
\caption{\label{table:stock} \small Stock return example: the factor coefficients showing how the stock returns load on the estimated factors, and the $p$-values for testing the associations between the estimated factors and the stock returns using the data in the last 26 weeks 
}
\centering
\begin{tabular}{lcccc}









           & \multicolumn{ 2}{c}{Reduced-rank regression} & \multicolumn{ 2}{c}{Robust reduced-rank regression} \\

           & coefficient &    $p$-value & coefficient &    $p$-value \\
   Walmart &      0$\cdot$46 &      0$\cdot$44 &      0$\cdot$36 &      0$\cdot$23 \\

     Exxon &     $-$0$\cdot$15 &      0$\cdot$32 &      0$\cdot$14 &      0$\cdot$84 \\

        General Motors &      0$\cdot$96 &      0$\cdot$42 &      0$\cdot$90 &      0$\cdot$02 \\

      Ford &      1$\cdot$20 &      0$\cdot$64 &      0$\cdot$59 &      0$\cdot$18 \\

        General Electric &      0$\cdot$24 &      0$\cdot$67 &      0$\cdot$32 &      0$\cdot$06 \\

Conoco Phillips &    $-$0$\cdot$04 &      0$\cdot$19 &      0$\cdot$36 &      0$\cdot$08 \\

 Citi Group &      0$\cdot$27 &      0$\cdot$93 &      0$\cdot$45 &      0$\cdot$00 \\

       International Business Machines &      0$\cdot$36 &      0$\cdot$42 &      0$\cdot$57 &      0$\cdot$13 \\

       American International Group &      0$\cdot$19 &      0$\cdot$01 &      0$\cdot$58 &      0$\cdot$00 \\
\end{tabular}
\end{table}

In this work, we deem explicit outlier detection to be as  important  as robust low-rank estimation. Indeed,  the reduced-rank component may not  be of direct interest in some applications, as it often represents common background information shared across the response variables, while capturing  unusual    changes or jumps  is   helpful. The robustification of low-rank matrix estimation is non-trivial. A straightforward idea might be  to use  a robust loss function $\rho$ in place of   the squared error loss in \eqref{rrr}, leading to
\begin{align}
\min_{\bsbB}  \sum_{i=1}^{n} \rho( \|\bsbGamma^{1/2}(\bsby_i - \bsbB^\T\bsbx_i) \|_2) \qquad  \mbox{subject to}  \ r(\bsbB) \leq r,  \label{rob-rhorow}
\end{align}
but  such an estimator may  be difficult to compute. To the best of our knowledge, even when $\rho$ is  Huber's loss function \citep{Huber1981}, there is no   algorithm for solving \eqref{rob-rhorow}, let alone those nonconvex losses which are known to be more effective in dealing with multiple gross outliers with possibly high leverage values.   Another motivation   is that  nonasymptotic   theory  on the topic is  limited.     Classical robust analysis, ignoring   the low-rank constraint,  falls in either deterministic worst-cases studies,  or large-$n$   asymptotics  with $p$ and $m$  held fixed, which  may not  meet   modern   needs.

We propose a novel robust reduced-rank regression method for concurrent robust modeling and outlier identification. We explicitly introduce a sparse mean-shift outlier component and formulate a shrinkage  multivariate regression  in place of \eqref{rob-rhorow}, where  $p$ and/or $m$ can be  much larger than $n$.  The robust reduced-rank regression provides a general framework  and includes  M-estimation  and principal component pursuit    \citep{Huber1981,Hampel1986,ZhouCandes2010,candes2011}. It is worth mentioning that all the techniques developed in this work apply to high-dimensional sparse regression with a single response. In Section \ref{sec:rquad}, we show that low-rank estimation can be ruined by a single rogue point, and propose a robust reduced-rank estimation framework. A universal connection between the proposed robustification  and   conventional   M-estimation  is established,  regardless of the size of $p$, $m$ or $n$.   Section \ref{sec:theory} performs finite-sample theoretical studies  of   the proposed robust estimators, with the intention of    pushing     classical robust analysis to  multivariate data with possible large $p$ and/or $m$.    A computational algorithm    developed in Section \ref{sec:ipod} is easy to implement and leads to a coordinatewise minimum point with theoretical guarantees.     Section \ref{sec:app} shows some real applications.  
All proofs and simulation studies  are given in the Appendices. 

The following notation and symbols will be used throughout the paper.  We denote by    $\mathbb{N}$    the set of natural numbers. We use $a\wedge b$ to denote $\min(a, b)$ and $e$ to denote the Euler constant. Let $[n] = \{ 1, \ldots, n\}$. Given any matrix $\bsbA$,     $\Proj_{\bsbA}$  denotes    the orthogonal projection matrix onto the  range of $\bsbA$, i.e., $\bsbA ( \bsbA^\T \bsbA)^- \bsbA^\T$, where   $^-$ stands for the Moore--Penrose pseudoinverse. When there is no ambiguity, we also use     $\Proj_{\bsbA}$   to denote the column space of $\bsbA$. Let    $\|\bsbA\|_{\tF}$ denote  the Frobenius norm, $\|\bsbA\|_2$ denote  the spectral norm,   $\| \bsbA\|_0=\|\vect(\bsbA)\|_0=|\{(i, j): \bsbA(i,j)\neq 0\}|$ with $|\cdot|$  denoting the cardinality of the enclosed set. For $\bsbA=(\bsb{a}_1 \ \ldots \ \bsb{a}_n)^\T\in \mathbb R^{n\times m}$,  $\|\bsbA\|_{2,1}=\sum_{i=1}^n  \| \bsb{a}_i\|_2$, and $\|\bsbA\|_{2,0}=\sum_{i=1}^n 1_{\| \bsba_i\|  \ne 0}$ which gives    the number of non-zero rows of $\bsbA$. Given $\mathcal J \subset [n]$, we often denote $\sum_{i \in \mathcal J} \|\bsb{a}_i\|_2$ by  $\|\bsbA_{\mathcal J}\|_{2,1}$.  Threshold functions  are defined as follows.
   \begin{definition}[Threshold function]\label{def:threshold}
A threshold function is a real-valued
function  $\Theta(t;\lambda)$ defined for $-\infty<t<\infty$
and $0\le\lambda<\infty$ such that
(i) $\Theta(-t;\lambda)= -\Theta(t;\lambda)$;
(ii) $\Theta(t;\lambda)\le \Theta(t';\lambda)$ for $t\le t'$; 
(iii) $\lim_{t\to\infty} \Theta(t;\lambda)=\infty$;
(iv) $0\le \Theta(t;\lambda)\le t$\ for\ $0\le t<\infty$.
\end{definition}

\begin{definition}[Multivariate Threshold function]\label{def:thresholdmulti}
Given any $\Theta$,   $\vec\Theta$    is    defined  for any vector $\bsb{a}\in \mathbb R^{m}$ such that   $\vec\Theta(\bsb{a};\lambda)=
\bsb{a} \Theta(\|\bsb{a}\|_2;\lambda)/\|\bsb{a}\|_2$ for $\bsb{a}\neq  \bsb{0}$ and $\bsb{0}$ otherwise. For any matrix $\bsbA=( \bsb{a}_1 \  \ldots \ \bsb{a}_n  )^\T\in {\mathbb R}^{n\times m}$, $\vec\Theta(\bsbA;\lambda) = \{   \vec\Theta({\bsb{a}}_1;\lambda) \ \ldots \ \vec\Theta({\bsb{a}}_n;\lambda)  \}^\T$.
\end{definition}






\section{Robust Reduced-Rank Regression}\label{sec:rquad}


\subsection{Motivation}\label{sec:rquad:motivation}
 Although reduced-rank regression is associated with a   highly nonconvex problem  \eqref{rrr}, a global minimizer $\hat{\bsbB}$ can be obtained in explicit form. Given any $r$ ($1\leq r\leq \min(m,q)$) with $q=r(\bsbX)$,
\begin{align}
 \hat \bsbB (r) = \mathcal R(\bsbX, \bsbY, \bsbGamma, r) =  (\bsbX^\T \bsbX)^- \bsbX^\T \bsbY \bsbGamma^{1/2}\Proj_{\bsbV(\bsbX, \bsbY, \bsbGamma, r)} \bsbGamma^{-1/2}, \label{rrr-sol}
\end{align}
where    $\bsbV(\bsbX, \bsbY, \bsbGamma, r)$  is formed by the leading $r$ eigenvectors of $\bsbGamma^{1/2}\bsbY^\T \Proj_{\bsbX} \bsbY\bsbGamma^{1/2}$. See, e.g.,  \citet{reinsel1998} for a detailed     justification.
When $\bsbGamma= \bsbI$, we abbreviate $ \mathcal R(\bsbX, \bsbY, \bsbI, r)$ to $ \mathcal R(\bsbX, \bsbY,   r)$. The reduced-rank regression estimator is   denoted by  $\hat \bsbB(r)$ to emphasize its dependence on the regularization parameter.

Outliers are unavoidable in real   data.  We  define the finite-sample breakdown point for an arbitrary estimator $\hat \bsbB$, in the spirit of \cite{donoho1983bd}: given finite  data $(\bsbX, \bsbY, \bsbGamma)$  and an estimator   $\hat \bsbB(\bsbX, \bsbY, \bsbGamma)$, its breakdown point is
\begin{align*}
\epsilon^*(\hat\bsbB)= \frac{1}{n}\min\Big\{k\in \mathbb{N}\cup\{0\}: \sup_{\tilde \bsbY\in \mathbb R^{n\times m}: \|\tilde \bsbY - \bsbY \|_0 \leq k} \| \bsbX \hat\bsbB(\bsbX, \tilde\bsbY, \bsbGamma)\|_{\tF}=+\infty\Big\}.
\end{align*}

In addition to  the reduced-rank regression estimator $\hat\bsbB(r)$, we   take into account a general low-rank estimator obtained by imposing a singular value penalty
\begin{align}
\hat \bsbB(\lambda)\in \arg\min_{\bsbB} \frac{1}{2} \tr\{(\bsbY - \bsbX \bsbB) \bsbGamma (\bsbY - \bsbX \bsbB)^\T\} + \sum_{s=1}^{p\wedge m}P(\sigma_s^{ \bsbB \bsbGamma^{1/2} }; \lambda).\label{eq:grrr}
\end{align}
Here, $\lambda$ is a regularization parameter, and $\sigma_s^{ \bsbB \bsbGamma^{1/2}}$  denote the singular values of $\bsbB\bsbGamma^{1/2}$. The penalty $P$ is constructed from an arbitrary thresholding rule $\Theta(\cdot;\lambda)$    by
\begin{equation}
P(t;\lambda) - P(0;\lambda)=P_{\Theta}(t; \lambda) +q(t;\lambda), \ \  P_{\Theta}(t; \lambda)=\int^{|t|}_0 [\sup \{s:\Theta(s;\lambda) \le u\}-u]\rd u,
 \label{eq:construction}
\end{equation}
for some nonnegative  $q(\cdot; \lambda)$  satisfying $q\{\Theta(s;\lambda);\lambda\}=0$,   for all $  s\in \mathbb{R}$. 

\begin{theorem}\label{th:bp}
Given any finite  $(\bsbX, \bsbY, \bsbGamma)$ and $r\geq 1$ with $\bsbGamma$   positive definite and $\bsbX\neq \bsb{0}$, let  $\hat\bsbB(r)$ be a reduced-rank regression estimator which solves \eqref{rrr}. Then its finite-sample breakdown point  is exactly $1/n$. Furthermore, for any   $\hat \bsbB(\lambda)$ given by   \eqref{eq:grrr}, $\epsilon^*\{\hat \bsbB(\lambda)\}=1/n$ still holds for any finite value of $\lambda$.


\end{theorem}

The result indicates that   a single outlier can completely ruin   low-rank matrix estimation, whether one applies a  rank constraint  or,    say, a         Schatten $p$-norm penalty. 
The conclusion   limits the use of ordinary rank reduction  in big data applications. Because with the  low-rank constraint,  directly   applying a robust  loss function, as in \eqref{rob-rhorow}, may result in nontrivial computational  and  theoretical challenges, we will apply a novel  additive robustification, motivated by \cite{she2011a}.

\subsection{The  additive framework  }

We introduce a multivariate mean-shift regression model to explicitly encompass   outliers,
\begin{equation}
\bsbY = \bsbX\bsbB^* + \bsbC^* + \bsbE, \label{model2}
\end{equation}
where $\bsbB^* \in \mathbb{R}^{p\times m}$ gives the matrix of  coefficients, $\bsbC^{*} \in \mathbb{R}^{n\times m}$ describes  the outlying effects on $\bsbY$, and $\bsbE  \in\mathbb{R}^{n\times m}$ has   independently and identically distributed rows following  $N(\bsb0,\bsbSig)$. Obviously, this leads to an over-parameterized model, so we   must  regularize the unknown matrices appropriately. We assume that $\bsbB^*$ has low rank and $\bsbC^*$ is a sparse matrix with only a few  nonzeros because  outliers  are inconsistent with the majority of the data.
Given a positive definite weighting matrix $\bsbGamma$, we propose the   robust reduced-rank regression problem
\begin{align}
\min_{\bsbB, \bsbC} \  \frac{1}{2} \tr\{(\bsbY - \bsbX \bsbB - \bsbC) \bsbGamma (\bsbY - \bsbX \bsbB - \bsbC)^\T\}  + P(\bsbC;\lambda) \quad  \mbox{subject to} \  \ r(\bsbB) \leq r\mbox{.} \label{eq:r4Gamma}
\end{align}
Here, $P(\cdot; \lambda)$ is a sparsity-promoting penalty function with $\lambda$ to adjust   the amount of shrinkage, but it can also be a  constraint, such as \eqref{eq:constrR4}. The following form of $P$    can handle element-wise outliers
\begin{align}
P(\bsbC;\lambda)=\sum_{i=1}^n \sum_{k=1}^m P(|c_{i,k}|;\lambda),\label{elemwise}
\end{align}
which was used in  the  stock return analysis.   It is   more common in robust   statistics  to assume   outlying samples, or {outlying rows} in $(\bsbY,\bsbX)$,  which corresponds
to\begin{align}
P(\bsbC;\lambda) = \sum_{i=1}^n P(\|\bsbc_i\|_2; \lambda), \label{rowwise}
\end{align}
where $\bsbc_i^\T$ is the $i$th row vector of $\bsbC$. Unless otherwise specified, we consider row-wise outliers. But all our algorithms and analyses  after simple modification  can handle  element-wise outliers.


In the literature on reduced-rank regression, it is common to regard the weighting matrix $\bsbGamma$ as known \citep{reinsel1998,  yuan2007, izenman2008}. The choice of $\bsbGamma$ is flexible and is usually based on  a pilot covariance estimate  $\hat{\bsbSig}$. For example, it can be   $ \hat{\bsbSig}^{-1}$ when $\hat{\bsbSig}$ is nonsingular, or a regularized version $ (\hat{\bsbSig} + \delta\bsbI )^{-1}$ for some $\delta >0$. Although it   sounds intriguing to consider     jointly estimating    the high-dimensional   mean   and the even higher-dimensional covariance matrix in the presence of outliers, this  is beyond the scope of this paper.     When a reliable estimate of $\bsbSig$ is unavailable, a standard practice   in finance and econometric forecasting is   to reduce $\bsbGamma$ to a diagonal matrix, or equivalently, an identity matrix after robustly scaling the response variables.  
For ease of presentation, we take $\bsbGamma$ as the identity matrix unless otherwise noted, and mainly focus on the following robust reduced-rank regression criterion,
\begin{align}
\min_{\bsbB, \bsbC} \  \frac{1}{2}\|\bsbY - \bsbX \bsbB - \bsbC\|_{\tF}^2  + P(\bsbC;\lambda) \qquad  \mbox{subject to}  \ \ r(\bsbB) \leq r\mbox{.}\label{eq:r4}
\end{align}

We  show that the proposed  additive  outlier characterization   indeed comes with a robust guarantee, and interestingly, it generalizes    M-estimation    to the multivariate rank-deficient setting.
 We write $\bsbY= (\bsby_1, \ldots, \bsby_n)^\T$ and  $\bsbC = (\bsbc_1, \ldots, \bsbc_n)^\T$.

\begin{theorem}\label{th:rob}
(i) Suppose $\Theta(\cdot;\lambda)$ is an arbitrary {thresholding rule} satisfying    Definition \ref{def:threshold}, and let $P$  be any penalty associated with $\Theta$ through \eqref{eq:construction}.
Consider \begin{align}
\min_{\bsbB,\bsbC}\frac{1}{2} \|\bsbY - \bsbX \bsbB - \bsbC\|_{\tF}^2  + \sum_{i=1}^n P( \|\bsbc_i \|_2;\lambda) \ \ \mbox{subject to}  \ r(\bsbB) \leq r\mbox{.}  \label{eq:penR4}
\end{align}
For any fixed $\bsbB$, a globally optimal solution for $\bsbC$ is $\bsbC(\bsbB)=\vec\Theta(\bsbY - \bsbX \bsbB;\lambda)$. By profiling out $\bsbC$ with $\bsbC(\bsbB)$, \eqref{eq:penR4} can be expressed as an optimization problem with respect to $\bsbB$ only, and it is equivalent to the robust M-estimation problem
\begin{align}
\min_{\bsbB}\sum_{i=1}^n \rho( \| \bsby_i - \bsbB^\T\bsbx_i \|_2; \lambda) \  \ \mbox{subject to}  \ r(\bsbB) \leq r,\label{eq:penrob}
\end{align}
where the robust loss function $\rho$ is given by  $$\rho(t;\lambda) = \int_0^{|t|} \psi(u;\lambda) \rd u, \quad \psi(t;\lambda) = t - \Theta(t;\lambda)\mbox{.}$$

(ii) Given $\outl\in \{0, 1, \ldots, n\}$, consider
\begin{align}
\min_{\bsbB,\bsbC}\frac{1}{2} \|\bsbY - \bsbX \bsbB - \bsbC \|_{\tF}^2   \  \ \mbox{subject to}  \ r(\bsbB) \leq r, \|\bsbC\|_{2,0} \leq \outl\mbox{.} \label{eq:constrR4}
\end{align}
 Similarly,         \eqref{eq:constrR4}, after   profiling out $\bsbC$,  can be expressed as    an optimization problem with respect to $\bsbB$ only, and  is equivalent to the rank-constrained trimmed least squares problem
\begin{align}
\min_{\bsbB}\frac{1}{2} \sum_{i=1}^{n-\outl} r_{(i)}    \quad  \mbox{subject to} \  \  r(\bsbB) \leq r, r_i = \|\bsby_i - \bsbB^\T\bsbx_i \|_2, \label{eq:constrrob}
\end{align}
where $r_{(1)}, \ldots, r_{(n)}$ are the order statistics of $r_1, \ldots, r_n$ satisfying $| r_{(1)} | \le \cdots \le |r_{(n)}|$.
\end{theorem}

\remark \label{remRob1} \upshape
Theorem \ref{th:rob} connects   $P$ to $\rho$ through    $\Theta$. As is well known, changing the squared error loss  to  a robust loss amounts to  designing a set of  {multiplicative}  weights for $\bsby_i - \bsbB^\T \bsbx_i$ ($i=1,\ldots, n$). Our   additive robustification   achieves the   same robustness, but leaves the original loss function untouched. The  connection is also valid in the case of  element-wise outliers,  with  $P$ and $\rho$  applied in an element-wise manner. In fact,    the identity built  in Lemma \ref{thresh-identity} in the Appendices,
$$
\frac{1}{2} \{r - \Theta(r; \lambda)\}^2 + P_{\Theta}\{\Theta(r; \lambda); \lambda\}  = \int_0^{|r|} \psi(t; \lambda) \rd t, \quad r\in \mathbb R,
$$
 implies that  the equivalence   holds  much more  generally,      with $\bsbB$ subject to an arbitrary    constraint or penalty, and   regardless of the number of  response variables and the number of predictors. This extends the main result in \cite{she2011a} to multiple-response models with $p$ possibly larger than $n$. 

\remark \label{remRob2} \upshape
Theorem \ref{th:rob} holds for {all} thresholding rules, and     popularly-used convex and nonconvex penalties   are all covered by  \eqref{eq:construction}.
  For example,  the convex  group $\ell_1$  penalty $\lambda \sum \|\bsbc_i \|_2$ is associated with the soft-thresholding $\Theta_S(s;\lambda)=\mbox{sgn}(s) (|s|-\lambda)_+$.   The group $\ell_0$ penalty  $({\lambda^2}/2)\sum_{i=1}^n  1_{\|\bsbc_i\|_2\neq 0}$  can be obtained from  \eqref{eq:construction} with the hard-thresholding ${\Theta}_H(s;\lambda)= s 1_{|s|>\lambda}$, and  $q(t;\lambda)= 0.5 (\lambda - |t|)^2 1_{0<|t|<\lambda}$.
Our $\Theta$-$P$ coupling framework also covers
    $\ell_{p}$ ($0<p<1$), the smoothly clipped absolute deviation penalty \citep{fan2001}, the minimax concave penalty    \citep{zhang2010},   and the capped $\ell_1$ \citep{Zhang2010capped} as particular instances; see   \cite{She2012}. 


\remark \label{remRob3} \upshape
The universal link between  \eqref{eq:penR4} and \eqref{eq:penrob}    provides   insight  into the choice  of regularization. It is easy to verify that  the      $\ell_1$-norm penalty  as  commonly used in variable selection   leads  to   Huber's loss, which   is   prone to masking and swamping   and may   fail with even  moderately  leveraged outliers occurring.  To  handle gross outliers, redescending $\psi$ functions are often advocated,  which   amounts to using nonconvex penalties in \eqref{eq:penR4}. For example,  Hampel's three-part $\psi$ \citep{Hampel1986} can be shown to give Fan and Li's   smoothly clipped absolute deviation penalty, the skipped mean $\psi$ corresponds to the exact $\ell_0$ penalty, and     rank constrained least trimmed squares   can be rephrased as the $\ell_0$-constrained form as in \eqref{eq:constrR4}. Our approach  not only provides a unified way to robustify low-rank matrix estimation, but facilitates theoretical analysis and  computation of reduced-rank M-estimators in high dimensions. 




\subsection{Connections and extensions}

Before we dive into   theoretical study, it is worth pointing out some     connections and extensions of the proposed  framework. 
First, one can  set  $\bsbGamma$  equal to the inverse covariance matrix of the   response variables to perform   robust  canonical correlation analysis; see  \cite{reinsel1998}. Although we mainly focus on the rank-constrained form, there is no difficulty in extending our discussion to
\begin{align}
\min_{\bsbB, \bsbC}  & \frac{1}{2}\|\bsbY - \bsbX \bsbB - \bsbC\|_{\tF}^2 + \sum_{s=1}^{p\wedge m}  P_B(\sigma_s^{\bsbB};  \lambda_B) + P_C(\bsbC;\lambda_C),
\label{doubpenrquad}
\end{align}
where $\sigma_s^{\bsbB}$ denote the singular values of $\bsbB$, and    $ P_B$ and $P_C$ are sparsity-inducing penalties.

Our   robust reduced-rank regression subsumes   a special but important case,  $\bsbY = \bsbB + \bsbC + \bsbE$. This problem is perhaps less challenging than its supervised counterpart, but has wide applications in computer vision and machine learning  \citep{Wright2009,candes2011}.

Finally,       our method   can be extended to reduced-rank    generalized linear models; see, e.g.,  \citet{yee2003} and \cite{She2013} 
for some computational details. In these scenarios,   directly robustifying the  loss      can be messy, but   a sparse    outlier  term can always be   introduced without altering  the form of the given loss, so that many algorithms designed for fitting ordinary generalized linear models can be seamlessly applied.   


%
%


%




\section{Nonasymptotic Robust Analysis}\label{sec:theory}

Theorem \ref{th:rob} provides robustness and    some helpful intuition for the proposed method, but it might not be enough from a theoretical point of view. For example, can one   justify  the need for     robustification in estimating a matrix of low rank? Is using redescending $\psi$ functions still preferable in rank-deficient settings? Different from     traditional robust analysis, we cannot assume  an infinite sample size and  a  fixed number of predictors or response variables, because $p$ and/or $m$ can be much larger than $n$ in modern   applications. Conducting nonasymptotic robust analysis would be  desirable. The  finite-sample results  in this section  contribute to this type of robust analysis.


For simplicity we assume that the model is given by $\bsbY = \bsbX \bsbB^* + \bsbC^* + \bsbE$, where $\bsbE$ has independent and identically distributed      $   N(0, \sigma^2)$ entries,  and  consider the robust reduced-rank regression problem defined in \eqref{eq:r4}.
The noise distribution can be more general. For example, in all the following theorems except Theorem \ref{th_minimax}, $\bsbE$ can be sub-Gaussian. Given an estimator     $(\hat \bsbB, \hat \bsbC)$, we   focus on  its prediction accuracy measured by $M(\hat{\bsbB}-\bsbB^*,\hat{\bsbC}-\bsbC^*)$, where
\begin{equation}
  M(\bsbB,\bsbC)=\|\bsbX\bsbB+\bsbC\|_{\tF}^2\mbox{.}
\end{equation}
This predictive learning perspective is always legitimate  in evaluating the performance of an estimator, and requires no      signal strength or model uniqueness assumptions. The $\ell_2$-recovery  of $M(\hat{\bsbB}-\bsbB^*,\hat{\bsbC}-\bsbC^*)$ is fundamental, and such a bound, together with additional regularity assumptions,  can be easily adapted to obtain  estimation error bounds in different norms as well as selection consistency   \citep{ye2010,lounici2011}; see Theorem \ref{th_est}  in the Appendices for instance. Given a penalty function $P$, or equivalently, a robust loss $\rho$, we will study the performance of the set of global minimizers to show the ultimate power of the associated method. But our proof techniques apply more generally; see, e.g., Theorem \ref{compconv}. 

For any $\bsbC= ( \bsbc_1, \ldots, \bsbc_n)^\T$, define
\begin{align}\mathcal J(\bsbC)= \{i: \bsb{c}_i \neq \bsb{0}\}, \qquad  J(\bsbC)=|\mathcal J(\bsbC)|=\|\bsbC\|_{2,0}\mbox{.}
\end{align}
We use $r^* = r(\bsbB^*)$ to denote  the rank   of the true coefficient matrix, and $J^* = J(\bsbC^*)$ to denote the number of nonzero rows in $\bsbC^*$, i.e., the number of outliers.
Let $q = r(\bsbX)$.

To address   problems in arbitrary  dimensions, we build some finite-sample oracle inequalities \citep{donoho1994}.
The first  theorem considers a general penalty $P(\bsbC;\lambda)= \sum_{i=1}^n P(\| \bsbc_i\|_2; \lambda)$. Here, we assume that    $P(\cdot; \lambda)$  takes  $\lambda$ as the threshold parameter, and satisfies
\begin{align}
P(0; \lambda)= 0, \,\mbox{   } \,
P(t; \lambda) \geq P_H(t; \lambda),\label{cond:pen}
\end{align}
where $P_H(t; \lambda)= (-t^2/2+\lambda |t|)1_{|t|<\lambda} +(\lambda^2/2) 1_{|t|\geq \lambda}$. The latter inequality   is natural in view of \eqref{eq:construction}, because a shrinkage  estimator  with $\lambda$ as the threshold  is always bounded above by the hard-thresholding function $\Theta_H(\cdot, \lambda)$.   From Theorem \ref{th:rob}, \eqref{cond:pen} covers all $\psi$-functions bounded below by the skipped mean $\psi_H(s;\lambda)=s 1_{|s|\leq \lambda}$ for any $s\geq 0$.

\begin{theorem}\label{th_oracle}
Let $\lambda =A\sigma (m+\log n)^{1/2}$ with $A$   a  constant and  let    $(\hat\bsbB, \hat\bsbC)$ be  a global minimizer of \eqref{eq:r4}. Then, for any sufficiently large $A$, the following  oracle  inequality holds for  any $(\bsbB, \bsbC)\in \mathbb R^{p\times m}\times \mathbb R^{n\times m}$ satisfying $r(\bsbB) \leq r$:
\begin{align}
 \EE \{ M(\hat \bsbB -\bsbB^*,  \hat \bsbC - \bsbC^* ) \}\lesssim    M( \bsbB -\bsbB^*, \bsbC - \bsbC^*)
 + \sigma^2 (q+ m) r + P(\bsbC; \lambda)+\sigma^{2},\label{genoracle}
\end{align}
where   $\lesssim$ means the inequality holds up to a multiplicative  constant.

\end{theorem}

\begin{corollary} \label{l0oracle}
 Under the same conditions of Theorem \ref{th_oracle},   if $r\ge 1$ and  $P$ is   a {bounded nonconvex} penalty  satisfying   $P(t;\lambda) \lesssim   \lambda^2$  for any $t \in \mathbb R$, we have    \begin{align}
\begin{split}
 &\,\,\ \EE \{ M(\hat \bsbB -\bsbB^*,  \hat \bsbC - \bsbC^* )\}\lesssim       \\  &  \ \ \inf_{(\bsbB, \bsbC): r(\bsbB)\leq r}\{ M( \bsbB -\bsbB^*, \bsbC - \bsbC^*)
 + \sigma^2 (q+ m) r + \sigma^2 J(\bsbC) m + \sigma^2 J(\bsbC) \log n \}\mbox{.}
\end{split}\label{rateoracle}
\end{align}
\end{corollary}


\remark \label{remTh1} \upshape
Both  \eqref{genoracle} and \eqref{rateoracle}  involve a {bias} term $ M( \bsbB -\bsbB^*$, $\bsbC - \bsbC^*)$.
 Setting  $r=r^*$,   $\bsbB = \bsbB^*$ and $\bsbC=\bsbC^*$  in, say, \eqref{rateoracle},
we obtain  a prediction error bound   of the order
\begin{align}
\sigma^2 (q+ m) r^* + \sigma^2 J^*({m+\log n})\mbox{.} \label{essrate}
\end{align}
On the other hand, the presence of  the  bias term  ensures   applicability of   robust reduced-rank regression to weakly sparse $\bsbC^*$, and similarly,  $r$ may also deviate from $r^*$ to some extent, as a benefit from the bias-variance trade-off.

\remark \label{remTh2} \upshape
Our proof scheme can also be used to show similar conclusions for the doubly penalized form \eqref{doubpenrquad} and the  doubly constrained form \eqref{eq:constrR4},  under the general assumption that   the noise matrix has  sub-Gaussian marginal tails.    The following theorem shows the result for  \eqref{eq:constrR4} which is one of our favorable forms in practical data analysis.
\begin{theorem}\label{th_oracle_contr}
Let $(\hat\bsbB, \hat\bsbC)$ be  a solution to \eqref{eq:constrR4}. With the convention $0 \log 0=0$, we have
\begin{align*}
& \EE  \{ M(\hat \bsbB -\bsbB^*,  \hat \bsbC - \bsbC^*)\}\lesssim \\ &  \inf_{r(\bsbB) \leq r,J(\bsbC)\le \outl }   M( \bsbB -\bsbB^*, \bsbC - \bsbC^*)
 + \sigma^2\{ (q+ m) r + \outl m +\outl \log (e n /\outl)\} +\sigma^{2}.
\end{align*}
\end{theorem}

 Theorem \ref{th_oracle_contr}    reveals some      breakdown point information as a by-product.      Specifically, fixing     $\bar \bsbY = \bsbX \bsbB$,    we contaminate    $\bsbY$     in the set    $\mathcal B(\outl)=\{\bsbY \in \mathbb R^{n\times m}: \bsbY = \bar \bsbY  + \bsbC + \bsbE,     \|\bsbC\|_{2,0} \leq \outl\}$, where  $    \vect(\bsbE) $ is  sub-Gaussian and $\outl\in  \mathbb{N}\cup\{0\}$.  Given      any estimator  $(\hat \bsbB, \hat\bsbC)$ which implicitly depends on $\bsbY$, we      define its risk-based finite-sample breakdown point by  $
\epsilon^*(\hat \bsbB, \hat\bsbC)= (1/n)\times \min\{\outl: \sup_{\bsb{Y}\in \mathcal {B}( \outl)}   \EE    \{M(\hat \bsbB -\bsbB,  \hat \bsbC - \bsbC)  \}=+\infty\} $, where     the randomness of  the estimator  is well   accounted   by taking the expectation. Then, for the estimator defined by \eqref{eq:constrR4},  
it follows from   Theorem \ref{th_oracle_contr} that          $\epsilon^*\geq (\outl+1)/n$. 

We emphasize that neither  Theorem \ref{th_oracle} nor Theorem \ref{th_oracle_contr}     places any requirement on $\bsbX$,  in contrast to Theorem \ref{th_oracle-l1}. 

\remark \label{remTh3} \upshape
The benefit of applying a {re-descending} $\psi$ is  clearly shown by Theorem \ref{th_oracle}. As an example, for Huber's $\psi$, which corresponds to the popular  convex $\ell_1$ penalty due to  Theorem \ref{th:rob},   $P(\bsbC; \lambda)$ on the right hand side of \eqref{genoracle} is unbounded, while Hampel's three-part $\psi$ gives a finite rate as shown in \eqref{rateoracle}. 
Furthermore, we show that in a minimax sense, the  error rate obtained in Corollary \ref{l0oracle}  is essentially optimal.
Consider the    signal class
 \begin{gather}
\mathcal {S}(r, J)= \{(\bsbB^*, \bsbC^*): r(\bsbB^*) \leq r, J(\bsbC^*)\leq J   \}, \  {1} \leq J \leq n/2,{1} \leq r \leq q\wedge m\mbox{.}
\end{gather}
Let $\ell(\cdot)$ be a nondecreasing loss function with $\ell(0)=0$, $\ell \not\equiv0$. 

\begin{theorem}
\label{th_minimax}
Let $\bsb{Y}=\bsbX\bsbB^* + \bsbC^*+\bsbE$ where   $\bsbE$ has     independently and identically distributed  $N({0},\sigma^2)$ entries. Assume that     $n\geq 2$, $1\leq J \leq n/2$, $1\leq r \leq q\wedge m$, $r(q+m-r)\geq 8$,
   and  $\sigma_{\min}(\bsbX)/\sigma_{\max}(\bsbX)$ is a positive constant, where $\sigma_{\max}(\bsbX)$ and $\sigma_{\min}(\bsbX)$  denote  the largest and the smallest   {nonzero} singular values of $\bsbX$, respectively.    Then there exist positive constants $\tilde{c}$, $c$, depending on $\ell(\cdot)$ only, such that
\begin{equation}
    \inf_{(\hat \bsbB, \hat \bsbC)}\,\sup_{(\bsbB^*, \bsbC^*) \in \mathcal {S}(r, J)}\EE(\ell[M(\hat{\bsbB}-\bsbB^*,\hat{\bsbC}-\bsbC^*)/\{\tilde{c} P_o(J,r)\}]) \geq c>0, \label{minimaxlowerbound}
\end{equation}
where   $(\hat \bsbB, \hat \bsbC)$ denotes any  estimator of $(\bsbB^*, \bsbC^*)$ and
\begin{align}
P_o(J,r)= \sigma^2\{r(q + m) + Jm + J\log(en/J)\}.
\end{align}
\end{theorem}

  We give some examples of $\ell$ to illustrate the conclusion. Using the indicator function $\ell(u)=1_{u\geq 1}$,  for any estimator $(\hat{\bsbB},\hat{\bsbC})$,
$
M(\hat{\bsbB}-\bsbB^*,\hat{\bsbC}-\bsbC^*)  \gtrsim \sigma^2\{r(q + m) + Jm + J\log(en/J)\}
$
holds with positive probability. For $ \ell(u)=u$, Theorem \ref{th_minimax} shows that the risk $\EE\{M(\hat{\bsbB}-\bsbB^*,\hat{\bsbC}-\bsbC^*)\}$ is bounded from below by $P_o(J, r)$,  up to some multiplicative constant. Therefore, \eqref{essrate}    attains the minimax  optimal rate  up to a mild logarithm factor,    showing the advantage of  utilizing redescending $\psi$'s in robust low-rank estimation. The analysis is nonasymptotic and applies to any $n$, $p$, and $ m$. 


Convex methods are  not hopeless, however. In  some less challenging problems, where   some incoherence  regularity condition  is satisfied by the augmented design matrix, Huber's $\psi$ can  achieve the same low error rate. The result     of the following theorem can be extended to any sub-additive penalties with the associated  $\psi$ sandwiched by   Huber's $\psi$ and  $\psi_H$.



\begin{theorem}\label{th_oracle-l1}
Let $(\hat \bsbB, \hat\bsbC) = \arg \min_{(\bsbB, \bsbC)  } \| \bsbY - \bsbX \bsbB-\bsbC\|_{\tF}^2/2 + \lambda \|\bsbC\|_{2,1}$ subject to $r(\bsbB)\leq r$,    $\lambda =A\sigma (m+\log n)^{1/2}$ where  $A$ is a  large enough constant. Then
\begin{align}
\begin{split} \EE \{ M(\hat \bsbB -\bsbB^*,  \hat \bsbC - \bsbC^* ) \}\lesssim\,   & M( \bsbB -\bsbB^*, \bsbC - \bsbC^*)+ \sigma^2
 \\ &  + \sigma^2 (q+ m) r  +     \sigma^2K^2 J(\bsbC)(m +   \log n)   \label{eq:oraclel1K}
\end{split}
\end{align}     for  any   $(\bsbB, \bsbC)$ with $\mbox{rank}(\bsbB)\leq r$, if given     $\mathcal J= \mathcal J(\bsbC)$,    $\bsbX$ satisfies   $(1+\vartheta) \|\bsbC_{\mathcal J}'\|_{2,1}  \le \|\bsbC_{\mathcal J^c}'\|_{2,1}+ K  |\mathcal J| ^{1/2} \| (\bsbI - \Proj_r)    \bsbC'\|_{\tF}$  for all    $\bsbC'$ and $\Proj_r: \Proj_r \subset \Proj_{\bsbX}, r(\Proj_r)\le 2r$, where     $K\ge 0$   and       $\vartheta$ is a positive constant.  \end{theorem}

Compared with \eqref{rateoracle},
\eqref{eq:oraclel1K} has an additional factor of $K$ on the right-hand side.
Under a different regularity condition, an estimation error bound  on $\bsbB^*$ can be obtained. See Theorem \ref{th_est}. 
\remark \label{remTh5} \upshape
The  results obtained can be used to argue the necessity of robust   estimation when outliers   occur.   Similar to Theorem \ref{th_oracle}, we can show that the ordinary reduced-rank regression,  which sets $\hat \bsbC=\bsb0$,  satisfies
\begin{align}
 \EE \{M(\hat{\bsbB}-\bsbB^*,\hat{\bsbC}-\bsbC^*)\}\ &\lesssim   \inf_{  r(\bsbB) \leq r}    \|\bsbX  \bsbB -(\bsbX\bsbB^* + \bsbC^*) \|_{\tF}^2  + \sigma^2 (q+ m) r +\sigma^2\mbox{.} \label{rrrbnd}
\end{align}

Taking $r=r^*$, the error bound of the reduced-rank regression, evaluated at the optimal $\bsbB$   satisfying  $\bsbX \bsbB = \bsbX \bsbB^* + \Proj_{\bsbX \bsbB^*} \bsbC^*$ and $r(\bsbB)\le r$, is of order
\begin{align}
\sigma^2 (q+ m) r^* + \|(\bsbI - \Proj_{\bsbX \bsbB^*} )\bsbC^*\|_{\tF}^2\mbox{.} \label{rrrsimpbnd1}
\end{align}
Because $\bsbX \bsbB^*$ has   low rank,   $\bsbI - \Proj_{\bsbX \bsbB^*}$ is not null in general. Notable outliers that can affect the projection subspace in performing rank reduction tend to occur in the orthogonal complement   of the range of  ${\bsbX \bsbB^*}$, and so   \eqref{rrrsimpbnd1}  can be  arbitrarily large, which echoes the deterministic breakdown-point conclusion in Theorem \ref{th:bp}.

To control the size of the bias term, a better way    is to apply a larger rank value in the presence of outliers. Concretely, setting  $\bsbB=\bsbB^* + (\bsbX^\T\bsbX)^- \bsbX^\T \bsbC^*$ in \eqref{rrrbnd} yields
\begin{align}
\sigma^2J^{*} q+ \sigma^2  J^* m+\sigma^2 (q + m) r^* +  \| (\bsbI - \Proj_{\bsbX}) \bsbC^*\|_{\tF}^2, \label{rrrsimpbnd2}
\end{align}
where we used     $r(\bsbB) \leq r^* + J^*$.
When   $p> n$,  $\Proj_{\bsbX}=\bsbI$, and so \eqref{rrrsimpbnd2} offers an improvement  over  \eqref{rrrsimpbnd1} by  giving a finite error rate of  $\sigma^2J^{*} q+ \sigma^2  J^* m+\sigma^2 (q + m) r^* $. But our robust reduced-rank regression guarantees a consistently lower rate at $\sigma^2 J^* \log n+ \sigma^2  J^*m +\sigma^2 (q+ m) r^*  $, since $\sigma^2 J^* q \gg \sigma^2 J^* \log n$. The performance gain can be dramatic in big data applications, where the design matrix  is huge and typically multiple outliers are bound to occur.

\section{Computation and Tuning}\label{sec:ipod}

In this section, we show that compared with the M-characterization in    Theorem \ref{th:rob}, the additive formulation   \eqref{model2}     simplifies     computation and parameter tuning.  Let us consider a penalized form of the robust reduced-rank regression problem
\begin{align}
\min_{\bsbB, \bsbC} F(\bsbB, \bsbC)= \frac{1}{2} \|\bsbY - \bsbX \bsbB - \bsbC\|_{\tF}^2 + \sum_{i=1}^n P( \|\bsbc_i \|_2;\lambda) \  \mbox{subject to}  \ r(\bsbB) \leq r\mbox{.} \label{eq:penR4-again}
\end{align}
The penalties of interest  may be nonconvex in light of the theoretical results in Section \ref{sec:theory}, as        stringent incoherence assumptions  associated with convex penalties can be much relaxed or even removed.  Assuming that $P$ is constructed by   \eqref{eq:construction},
a simple algorithm  for solving \eqref{eq:penR4-again} is described as follows, where the two matrices $\bsbC$ and $\bsbB$ are alternatingly updated with the other held fixed until   convergence. Here, the multivariate thresholding,    $\vec\Theta$, is defined on basis of   $\Theta$, cf.  Definition \ref{def:threshold} and  Definition \ref{def:thresholdmulti}. 

\begin{algorithm}[!h]
\caption{A robust reduced-rank regression algorithm.} \label{al1}
\begin{tabbing}
   \enspace  Input $\bsbX$, $\bsbY$, $\bsbC^{(0)}$, $\bsbB^{(0)}$,   $\Theta$, $t=0$.\\
   \enspace Repeat \\
   \quad  (a) $t\leftarrow t+1$\\
   \quad  (b) $\bsbC^{(t+1)} \leftarrow \vec\Theta( \bsbY - \bsbX \bsbB^{(t)};\lambda)$\\
   \quad (c) $\bsbB^{(t+1)} \leftarrow    \mathcal R (\bsbX, \bsbY - \bsbC^{(t+1)},  r)$, as defined in \eqref{rrr-sol}\\
   \enspace Until convergence. 
\end{tabbing}
\end{algorithm}

Step (b) performs simple multivariate thresholding operations and    Step (c) does  reduced-rank regression   on the adjusted response matrix $\bsbY- {\bsbC}^{(t+1)}$. We do not really have to explicitly compute    $\bsbB$    to   update $\bsbC$ in the iterative process. In fact, only $\bsbX \bsbB^{(t)}$ is needed, which   depends on $\bsbX$ through $\Proj_{\bsbX}$, or   $\bsbI$ when $p\gg n$. The eigenvalue decomposition   called in \eqref{rrr-sol}    has   low computational  complexity because      the rank values of practical interest are often   small. Algorithm \ref{al1} is   simple to implement and is cost-effective.  For example, even for $p=1200$ and $n=m=100$, it takes only about 40 seconds to compute a whole solution path for a two-dimensional grid of 100 values of $\lambda$  and 10 rank values.

\begin{theorem} \label{compconv}
Let $\Theta$ be an arbitrary thresholding rule, and    $F$ be defined in \eqref{eq:penR4-again}, where $P$ is associated with $\Theta$ through  \eqref{eq:construction}. Then given any $\lambda\ge 0$ and $r\ge 0$, the proposed algorithm has the property that $F(\bsbB^{(t)}, \bsbC^{(t)})\ge F(\bsbB^{(t+1)}, \bsbC^{(t+1)}) $ for all $t $, and so   $F(\bsbB^{(t)}, \bsbC^{(t)})$ converges as $t\rightarrow \infty$. Furthermore, under the assumptions that  $\vec\Theta(\cdot; \lambda)$   is continuous  in the closure of     $ \{ \bsbY - \bsbX \bsbB^{(t)} \}$ and $\{\bsbB^{(t)}\}$ is uniformly bounded,       any accumulation point  of $(\bsbB^{(t)},  \bsbC^{(t)})$ is a  coordinatewise minimum point, and  a stationary point  when $q(\cdot; \lambda)\equiv 0$,  and hence  $F(\bsbB^{(t)}, \bsbC^{(t)})$ converges monotonically to $F(\bsbB^*, \bsbC^*)$ for some coordinatewise minimum point $(\bsbB^* , \bsbC^*)$.
\end{theorem}

The algorithm can be slightly modified to deal with     \eqref{elemwise},   \eqref{eq:constrR4},  and \eqref{doubpenrquad}. For example, we can replace  $\vec\Theta$   by $\Theta$, applied componentwise, to handle    element-wise outliers. 
 The   $\ell_0$-penalized form
with $P(\bsbC;\lambda)= ( {\lambda^2}/{2}) \|\bsbC\|_{2,0}$,
as well as  the constrained form \eqref{eq:constrR4}, will be used in data analysis and simulation. In implementation, they correspond to applying hard-thresholding and quantile-thresholding operators \citep{she2013b}.

 In common with most high breakdown  algorithms in robust statistics,   we recommend using the multi-sampling iterative strategy    \citep{rousseeuw1999fast}.
But  in many practical applications, we found  that the     initial values   can   be made rather freely. Indeed, Theorem \ref{th:algstat} shows that  if the problem is  regular, our algorithm  guarantees   low  statistical error   even without   the multi-start strategy.

In the following theorem, given   $\Theta$, define $
{\mathcal L}_{\Theta} = 1- \essinf\{ \rd \Theta^{-1}(u;\lambda)/\rd u: u \ge 0\}  ,
$
where    $\essinf$ is the   essential infimum. By definition, ${\mathcal L}_{\Theta}\le 1$.   We use $P_{2,\Theta}(\bsbC;\lambda)$ to denote $\sum_{i=1}^n P_{\Theta}(\| \bsbc_i\|_2;\lambda)$ for short and    set    $r = (1+\alpha) r^*  $ with $\alpha\ge 0$ and $r^*\ge 1$.  

\begin{theorem} \label{th:algstat}
Let $(\hat \bsbB, \hat \bsbC)$ be any solution satisfying  $\hat \bsbB =  \mathcal R(\bsbX, \bsbY - \hat \bsbC , r)$ and $\hat \bsbC = \vec\Theta( \bsbY - \bsbX \hat \bsbB ;\lambda)$ with $\hat \bsbB$ of rank $r$ and $\vec \Theta$   continuous at  $\bsbY - \bsbX \hat \bsbB $.  Let $\Theta$ be associated with a bounded nonconvex penalty as described in Corollary  \ref{l0oracle} and $\lambda =A\sigma (m+\log n)^{1/2}$ with  $A$     a  large enough constant.  Assume   that
$
     (1+\alpha)^{-1/2}\| \bsbX  \bsbB  - \bsbX  \bsbB^*\|_{\tF}^2 +  {\mathcal L_{\Theta}}  \| \bsbC  - \bsbC^*\|_{\tF}^2
 +   \vartheta P_{2,H } (\bsbC-\bsbC^*;\lambda) \le    ({2-\delta})  M( {\bsbB}-\bsbB^*, {\bsbC}-\bsbC^*) +  2P_{2,\Theta } (\bsbC;\lambda) + \zeta P_{2,0}  (\bsbC^*; \lambda)
$
holds for all  $(\bsbB, \bsbC)$   satisfying $r(\bsbB)\le r$, where $\zeta\ge 0$,   $\delta>0$ and $\vartheta>0 $ are   constants. Then
$
\EE \{M(\hat{\bsbB}-\bsbB^*,\hat{\bsbC}-\bsbC^*)\}  \lesssim \sigma^2 (1+\alpha)(q+ m) r^*  +     \sigma^2  J^*  m + \sigma^2 J^* \log n  $.
\end{theorem}

To choose an optimal rank for $\bsbB$ and an optimal row support for $\bsbC$ jointly,   cross-validation appears to be an   option.   However, it lacks theoretical support in    the robust low-rank setting, and for large-scale problems, cross-validation can be quite   expensive. Motivated by     Theorem \ref{th_minimax}, we propose the    predictive information criterion
\begin{align}
  \log\|\bsbY-\bsbX {\bsbB} - {\bsbC} \|_{\tF}^2   +\frac{1}{mn}[ A_1\{  J m + (m+q-r)r\} + A_2J \log (en  /    J)] ,\label{PIC}
\end{align}
where $\|\bsbY-\bsbX {\bsbB} - {\bsbC} \|_{\tF}^2$ is the residual sum of squared errors,  $r = r(\bsbB)$,   $  J=\|  \bsbC\|_{2,0}$, and recall that $e$ denotes the Euler constant.  The term $ J m + (m+q-r)r$  counts the degrees of freedom of the obtained model, and $J \log (en /   J)$ characterizes the risk inflation. The benefits of the criterion   include          no noise scale parameter needs to be estimated, and    minimizing \eqref{PIC}   achieves the minimax optimal  error  rate when the true model is parsimonious, as is shown below.
\begin{theorem}\label{logpic}
Let $P(\bsbB, \bsbC) =   J  m + (m+q-r)r  + J \log (en  /    J)$,  where   $r = r(\bsbB)$ and  $  J=\|  \bsbC\|_{2,0}$. Suppose that the true model is parsimonious in the sense that $ P(\bsbB^*, \bsbC^*)  < mn / A_0$ for some   constant $A_0>0$.    Let        $\delta(\bsbB, \bsbC) = AP(\bsbB, \bsbC) \allowbreak /(mn)$ where  $A$   is a positive constant satisfying $ A<A_0$, and so $\delta(\bsbB^*, \bsbC^*)<1$.     Then for sufficiently large values of  $\bsbA_0$ and $A$, any  $(\hat{\bsbB} , \hat\bsbC )$ that   minimizes  $\log \|\bsbY-\bsbX {\bsbB} - {\bsbC} \|_{\tF}^2 +\delta(\bsbB, \bsbC)$ subject to $\delta(\bsbB,\bsbC)<1$ satisfies $M(\hat{\bsbB}-\bsbB^*,\hat{\bsbC}-\bsbC^*) \lesssim \sigma^2 \{J^* m + (m+q-r^*)r^*  +J^{*} \log (en  /    J^*)\}$ with probability at least $1 - c_1' n^{- c_1 }-c_2'\exp(-c_2 m n )$ for some  constants $c_1, c_1', c_2, c_2'>0$.
\end{theorem}

Based on computer experiments, we set $A_1 = 7$, $A_2 = 2$.



\section{Arabidopsis Thaliana Data}\label{sec:app}

We performed extensive simulation studies to compare our method with some classical robust multivariate regression approaches and several reduced-rank methods    \citep{Tatsuoka2000,VanAelst2005,Roelant2009,reinsel1998, bunea2011,mukh2011} in both low  and high dimensions. The results are reported in the Appendices and show the excellent performance of the proposed method.

 Isoprenoids are abundant and diverse in plants, and they serve many important biochemical functions and have roles in respiration, photosynthesis and regulation of growth and development in plants. To examine the regulatory control mechanisms in the gene network for isoprenoid in Arabidopsis thaliana, a genetic association study was conducted, and with $n=118$ GeneChip microarray experiments performed  to   monitor   gene expression levels under various experimental conditions \citep{Wille2004}. It was experimentally verified that there exist strong connections between some downstream pathways and two isoprenoid biosynthesis pathways. We thus considered a multivariate regression setup, with the expression levels of $p=39$ genes from the two isoprenoid biosynthesis pathways serving as predictors, and the expression levels of $m=62$ genes from four downstream pathways, namely plastoquinone, caroteniod, phytosterol and chlorophyll, serving as the responses.

Because of the small sample size relative to the number of unknowns, we applied robust reduced-rank regression with the  predictive information criterion for  parameter  tuning. The final model has rank five, which reduces the effective number of unknowns   by about 80\% compared with the least squares model.  Interestingly, our method also identified  two outliers,  samples 3 and  52. Figure \ref{fig:path} shows the   detection paths by plotting     the $\ell_2$ norm of each   row in the $\bsbC$-estimates   for a sequence of  values of $\lambda$. The two unusual samples are   distinctive. The outlyingness might be caused by different experimental conditions. In particular, sample 3 was the only sample with Arabidopsis tissue culture in a baseline experiment.   The two outliers have a surprisingly big impact     on both coefficient estimation and model prediction. This can  be seen from    $\|\hat{\bsbB}-\tilde{\bsbB}\|_{\tF}/\|\tilde{\bsbB}\|_{\tF}\approx 50\%$, and  $\|\bsbX\hat{\bsbB}-\bsbX\tilde{\bsbB}\|_{\tF}/\|\bsbX\tilde{\bsbB}\|_{\tF}\approx 26\%$, where $\hat{\bsbB}$ and $\tilde{\bsbB}$ denote the robust reduced-rank regression and the plain reduced-rank regression estimates, respectively.
In addition, Figure \ref{fig:path} reveals that sample 27 could be a potential outlier which merits further investigation.

\begin{figure}[htp]
\centering
{\includegraphics[scale=0.7]{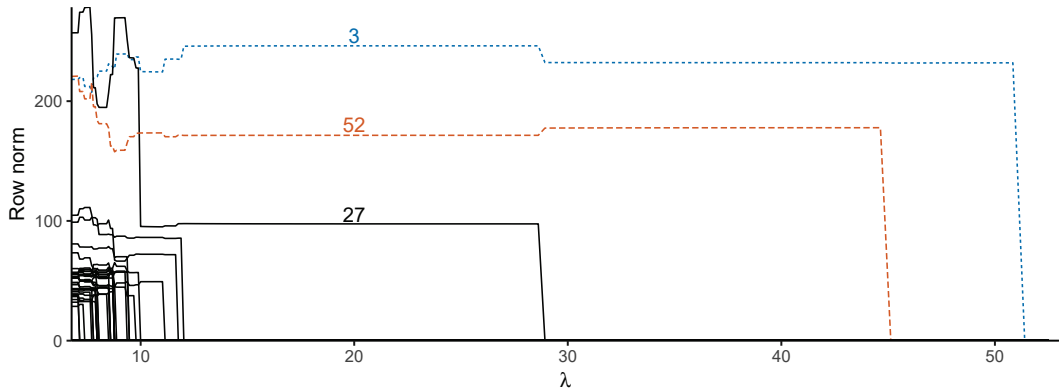}}
\caption{Arabidopsis thaliana data: outlier detection paths by the robust reduced-rank regression. Sample 3 and sample 52 are captured as outliers, whose paths are shown as a dotted line and a dashed line, respectively. The path plot also suggests sample 27 as a potential outlier.}\label{fig:path}
\end{figure}





The  low-rank model obtained    reveals  robust score variables, or factors, constructed from   isoprenoid biosynthesis
pathways, in response to the $62$  genes on the four downstream pathways. Let $\tilde{\bsbX}$ denote  the design matrix after removing the two detected outliers, and     $ \hat{\bsbU}\hat{\bsbD}\hat{\bsbV}^\T$ be  the singular value decomposition of $\tilde{\bsbX}\hat{\bsbB}$. Then $\hat\bsbU$ delivers five   orthogonal factors, and $\hat{\bsbV}\hat{\bsbD}$ gives the associated factor coefficients. Figure \ref{fig:loading} plots the   coefficients of the first three leading factors for all 62 response variables. 
Given the $s$th factor ($s=1,2,3$), the genes are grouped into the four pathways   separated by   vertical lines,   and   two horizontal lines are placed at heights $\pm \sigma_s^{ \tilde{\bsbX}\hat{\bsbB} }  {m}^{-1/2}$. Therefore, the genes located beyond the two horizontal lines   have relatively large coefficients     on the corresponding factor in magnitude.

\begin{figure}[htp]
\centering
{\includegraphics[width=\textwidth]{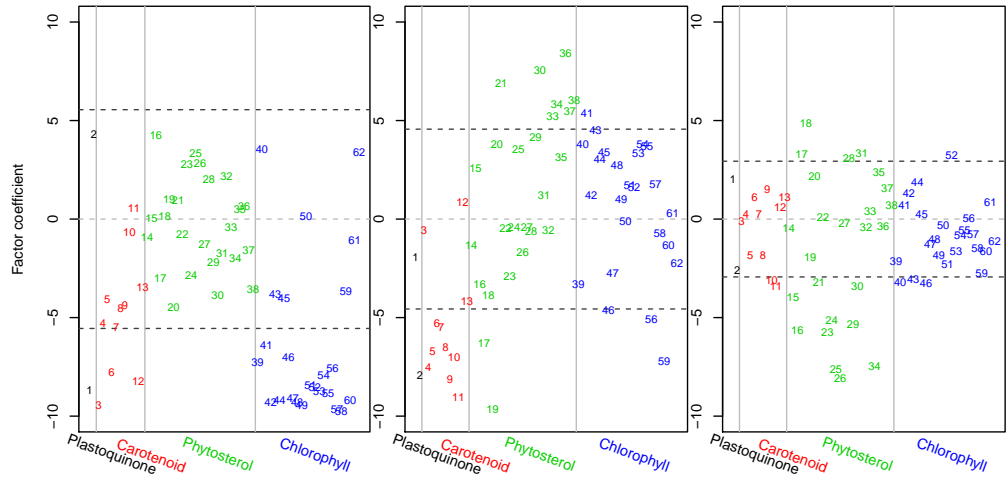}}
\caption{Arabidopsis thaliana data: factor coefficients of the 62 response genes from plastoquinone, caroteniod, phytosterol, and chlorophyll pathways. The panels from left to right correspond to the top three   factors estimated by the robust reduced-rank regression.  For the $s$th factor ($s=1,2,3$), two horizontal lines are plotted at heights $\pm \sigma_s^{ \tilde{\bsbX}\hat{\bsbB} }  {m}^{-1/2}$, and three vertical lines separate the genes into   different pathways.}\label{fig:loading}
\end{figure}

We also tested the significance of   the     factors in response to   each of the $62$ genes; see Table \ref{tab:5}.  Plastoquinone was excluded since it  has only two genes   and its behavior couples with that of caroteniod most of the time. Even   with    the  family-wise error rate controlled at $0.01$, the    factors obtained are overall  predictive according to the significance percentages, although they play very different roles in different pathways.
In fact,        according to Figure \ref{fig:loading} and  Table \ref{tab:5}, the genes that are correlated with the first factor are mainly from caroteniod and chlorophyll, and almost all the   coefficients there are   negative. It seems that the first factor   interprets   some joint characteristics of caroteniod and chlorophyll. The second factor  differentiates      phytosterol genes from     caroteniod genes, and the third factor seems to mainly contribute to   the phytosterol pathway.  Therefore,  by projecting the data onto a proper low-dimensional subspace in a supervised and robust manner, distinct behaviors of the downstream pathways and their potential subgroup structures can be revealed. More biological insights could be gained by closely examining the experimental and background conditions.





\begin{table}[t]
\centering
\small
\caption{\small Arabidopsis thaliana data: percentage of genes on each response pathway that show significance of a given factor, with the family-wise   error rate controlled at level $0.01$ 
}\label{tab:5}
\begin{tabular}{ccccc}

    Pathway      &   Number of genes           &   Factor 1 &   Factor 2 &   Factor 3 \\








Carotenoid &         11 &    55\% &    73\% &     9\% \\

Phytosterol &         25 &    20\% &   48\%&    32\%\\

Chlorophyl &         24 &   75\%&    21\% &     0\% \\

\end{tabular}
\end{table}

 \appendix
 \appendixpage
 \section{Proofs}
\label{sec:proofs}
\subsection{Notation and definitions}
\label{subsec:notation}
Given $\mathcal I\subset [n], \mathcal J\subset [p]$, $\bsbX(\mathcal I, \mathcal J)$ denotes a submatrix of $\bsbX$ by extracting  the rows and columns indexed by $\mathcal I$ and $\mathcal J$, respectively.
We use  $c$, $L$ to denote   constants. They are not necessarily the same at each occurrence. Denote by   $CS(\bsbA)$     the column space of $\bsbA$. Given     $\Proj_{\bsbA}$,  denote by $\Proj_{\bsbA}^{\perp}$ the projection onto  its orthogonal complement. In addition to the definitions of thresholding function $\Theta$ and the multivariate thresholding function $\vec\Theta$, we will use a matrix threshold function.



%

\begin{definition}[Matrix threshold function]\label{def:thresholdmat}
Given any  threshold function $\Theta(\cdot; \lambda)$, its matrix version $\Theta^{\sigma}$ is defined   for   $ \bsbB \in \mathbb R^{n\times m}$ as follows
\begin{eqnarray}
\Theta^{\sigma}(\bsbB;\lambda)= \bsbU\mbox{diag}\{\Theta(\sigma_{i}^{ \bsbB }; \lambda)\}\ \bsbV^\T,  \label{matthetadef}
\end{eqnarray}
where $\bsbU$, $\bsbV$, and $ \sigma_{i}^{ \bsbB }$ are obtained from the SVD of $\bsbB$: $\bsbB=\bsbU \mbox{diag}(\sigma_{i}^{ \bsbB }) \bsbV^\T$\mbox{.}
\end{definition}

Finally, we describe a {{quantile thresholding}} $\Theta^{\#}(\cdot; \outl, \eta)$ which   is convenient   in analyzing the  constraint-type problems.  It can be seen as a vector variant of the {hard-ridge} thresholding $\Theta_{HR}(t; \lambda, \eta)=t/(1+\eta) 1_{|t|> \lambda}$ \citep{She2009}. Given $1\le \outl \le  n$ and $\eta\ge 0$, $\Theta^{\#}(\bsb{a}; \outl, \lambda): {\mathbb R}^n\rightarrow {\mathbb R}^n$ is defined for any $\bsb{a}\in {\mathbb R}^n$ such that the $\outl$ largest components of $\bsb{a}$,  in absolute value, are shrunk   by a factor of $(1+\lambda)$ and the remaining components are all set to be zero.
In the case of ties, a random tie breaking rule is used. We abbreviate
    $\Theta^\#(\bsb{a}; \outl, 0)$ to  $\Theta^{\#}(\bsb{a}; \outl)$.

\subsection{Proof of Theorem 1}
We show  the proof detail for the penalized estimators. First,    the loss term in the objective can be decomposed  into   \begin{align*} \tr\{(\bsbY - \bsbX \bsbB) \bsbGamma (\bsbY - \bsbX \bsbB)^\T\} & = \|\bsbY \bsbGamma^{1/2} - \bsbX \bsbB\bsbGamma^{1/2}\|_{\tF}^2 \\ &=\|\Proj_{\bsbX}\bsbY \bsbGamma^{1/2} - \bsbX \bsbB\bsbGamma^{1/2}\|_{\tF}^2 + \|\Proj_{\bsbX}^{\perp} \bsbY \bsbGamma^{1/2}\|_{\tF}^2 \mbox{.}
\end{align*} Let $\bsbZ = \Proj_{\bsbX}\bsbY \bsbGamma^{1/2} $. Clearly, $\Proj_{\bsbZ} \subset \Proj_{\bsbX}$.
Consider the following  optimization problem
\begin{align}
\min_{\bsbA}\frac{1}{2} \|\bsbZ - \bsbA\|_{\tF}^2 + \sum _{s=1}^{p\wedge m}P(\sigma_s^{ \bsbA}; \lambda)\mbox{.} \label{subprob0}
\end{align}
 From   the proof of Proposition 2.1 in \cite{She2013},
the following results can be  obtained:   (i) any optimal solution $\hat \bsbA$ to \eqref{subprob0} must satisfy $\hat \bsbA \in \Proj_{\bsbZ} $;  (ii)  $ \bsbA_o=\Theta^\sigma(\bsbZ; \lambda)$ gives  a particular minimizer of \eqref{subprob0}, and  $\|\hat\bsbA - \bsbA_o \|_* \leq C(\lambda)$  holds for any  $\hat \bsbA$,  where $\| \cdot \|_*$ represents the nuclear norm and $C(\lambda)$ is a function dependent on the regularization parameter  only. From (i),     $\bsbX \hat\bsbB \bsbGamma^{1/2}$   is always      a solution to \eqref{subprob0}. It suffices to study the breakdown point of $\bsbA_o$.

Because $\bsbX \neq \bsb{0}$, there must exist $i\in [n]$ such that  the $i$th column of $\Proj_{\bsbX}$ is not $\bsb{0}$. Let $\tilde \bsbY = \bsbY + M \bsb{e}_i \bsb{e}_1^\T$.
where $\bsb{e}_i$ is the unit vector with the $i$th entry being  $1$. Due to the construction of $\tilde \bsbY$ and the positive-definiteness   of $\bsbGamma$,
$$\|\Proj_{\bsbX}\tilde \bsbY \bsbGamma^{1/2}\|_{\tF}^2=M^2 \|\Proj_{\bsbX} \bsb{e}_i \bsb{e}_1^\T \bsbGamma^{1/2}\|_{\tF}^2 + 2M \langle\Proj_{\bsbX} \bsbY, \bsb{e}_i \bsb{e}_1^\T \bsbGamma\rangle+ \|\Proj_{\bsbX}\bsbY \bsbGamma^{1/2} \|_{\tF}^2 \rightarrow +\infty$$
as $M\rightarrow \infty$.
 That is, given  $\lambda$,        $\Theta^\sigma (\Proj_{\bsbX} \tilde \bsbY \bsbGamma^{1/2} ; \lambda)$  thresholds the singular values of $\Proj_{\bsbX}\tilde \bsbY \bsbGamma^{1/2}$   the sum of which can be made arbitrarily large as $M$ increases. It follows from the definition of $\Theta$ that     $\sup_{M}\|\Theta^\sigma(\Proj_{\bsbX}\tilde \bsbY \bsbGamma^{1/2} ;\lambda)\|_{\tF}= \infty$.

The proof for the reduced-rank regression estimator follows similar  lines and is omitted.

\subsection{Proof of Theorem 2}
Part {(i)}: The proof of this part is based on the following two lemmas.
\begin{lemma}
\label{uniqsol-gen-grp}
Given an arbitrary thresholding rule $\Theta$ satisfying Definition \ref{def:threshold} in the paper,
let $P$ be any function associated with $\Theta$ through
\begin{align*}
P(t;\lambda) - P(0;\lambda)=P_{\Theta}(t; \lambda) +q(t;\lambda), \ \  P_{\Theta}(t; \lambda)=\int^{|t|}_0 [\sup \{s:\Theta(s;\lambda) \le u\}-u]\rd u,
\end{align*}
for some nonnegative $q(\theta; \lambda)$ satisfying  $q\{\Theta(t;\lambda)\}=0$ for all $t$.
Then, $\hat\bsbb=\vec\Theta(\bsby;\lambda)$ gives a globally optimal solution to
\begin{align*}
\min_{\bsbb \in {\mathbb R}^n}  \frac{1}{2}\|\bsby-\bsbb\|_2^2 + P(\|\bsbb\|_2;\lambda)\mbox{.}
\end{align*}
\end{lemma}

This result is implied by   Lemma 1 of \cite{She2012}. It is worth mentioning that      $\vec\Theta(\bsby;\lambda)$ is not necessarily   unique when $\Theta$ has discontinuities.
Next we prove an identity.
\begin{lemma}\label{thresh-identity} Given any thresholding rule  $\Theta(t; \lambda)$, define $P_{\Theta}(t; \lambda) =\int_0^{|t|} \{\Theta^{-1}(u;\lambda) - u\} \rd u$ where $\Theta^{-1}(u;\lambda)=\sup \{t:\Theta(t;\lambda) \le u\}$. Then the following identity holds for any $r\in\mathbb R$
\begin{align}
\frac{1}{2} \{r - \Theta(r; \lambda)\}^2 + P_{\Theta}\{\Theta(r; \lambda); \lambda\}  = \int_0^{|r|} \psi(t; \lambda) \rd t,
\end{align}
where $\psi(t; \lambda) = t - \Theta(t; \lambda)$.
\end{lemma}
\begin{proof}
Without loss of generality, assume $r\geq 0$. By    definition,   $ \int_0^{r} \psi(t; \lambda) \rd t =   r^2/2 - \int_0^r \Theta(t;\lambda) \rd t$ and $P_{\Theta}\{\Theta(r; \lambda); \lambda\}=\int_0^{\Theta(r;\lambda)} \Theta^{-1}(t;\lambda) \rd t-  r^2/2$. It suffices to show that $$\int_0^{\Theta(r;\lambda)} \Theta^{-1}(t;\lambda) \rd t + \int_0^r \Theta(t;\lambda) \rd t = r \Theta(r; \lambda)\mbox{.}$$
 In fact,   changing the order of integration, and using   the monotone property of $\Theta$,
 we get \begin{align*}
 \int_0^r \Theta(t;\lambda) \rd t - r \Theta(r; \lambda) &= \int_0^r \rd t \int_0^{\Theta(t;\lambda)} \rd s - \int_0^{\Theta(r;\lambda)} r \rd t\\
&=\int_0^{\Theta(r;\lambda)} \rd s \int_{ \Theta^{-1}(s;\lambda)}^{r} \rd t - \int_0^{\Theta(r;\lambda)} r \rd t \\
&= -\int_0^{\Theta(r;\lambda)} \Theta^{-1}(t;\lambda) \rd t\mbox{.}
\end{align*}
The conclusion thus follows.
\end{proof}

We   have the pieces in place to prove  part (i) of the theorem. Without loss of generality, assume $\bsbGamma=\bsbI$. Let $f(\bsbB, \bsbC)=  \tr\{(\bsbY - \bsbX \bsbB - \bsbC) (\bsbY - \bsbX \bsbB - \bsbC)^\T\}/2  + \sum_{i=1}^n P( \|\bsbGamma^{1/2}\bsbc_i   \|_2;\lambda)$, and $g(\bsbB) =\sum_{i=1}^n \rho( \| (\bsby_i  -  \bsbB^\T\bsbx_i )\|_2; \lambda)$. By Lemma \ref{uniqsol-gen-grp}, fixing $\bsbB$,   $\hat \bsbC=(\bsbc_1 \ \ldots \ \bsbc_n)^\T$ with  $\hat \bsbc_i=\vec\Theta(\bsby_i  -  \bsbB^\T\bsbx_i ; \lambda)$ gives an optimal solution to $\min_{\bsbC} f(\bsbB, \bsbC)$. For this $\hat \bsbC$, 
$f(\bsbB, \hat\bsbC) = g(\bsbB)$ holds by Lemma \ref{thresh-identity}. 

Part (ii): The proof follows similar lines of  that of Part (i), based on  the     quantile thresholding and Lemma C.1 in   \cite{she2013b}.
The details are  omitted.

\subsection{Proofs of Theorem 3 \& Theorem 6}
\label{proof_oraclel0l1}

Recall that  $P_1(t; \lambda) = \lambda |t|$,  $P_0(t; \lambda) =  ({\lambda^2}/{2}) 1_{t\neq 0}$, $P_H(t; \lambda)= (-t^2/2+\lambda |t|)1_{|t|<\lambda} +(\lambda^2/2) 1_{|t|\geq \lambda}.
$ For convenience,  $P_{2,1} (\bsbC; \lambda)$ is used to denote $\lambda \| \bsbC\|_{2,1}$, and $P_{2,0}$ and $P_{2, H}$  are used similarly.

 By definition,  $(\hat \bsbB, \hat \bsbC)$ satisfies the following inequality for any $(\bsbB, \bsbC)$ with $r(\bsbB)\leq r$,
\begin{align}
\frac{1}{2}M( \hat \bsbB -  \bsbB^*, \hat \bsbC -  \bsbC^*)\leq \frac{1}{2}M( \bsbB -  \bsbB^*, \bsbC -  \bsbC^*) +   P(\bsbC; {\lambda}) - P(\hat \bsbC; {\lambda})
+ \langle \bsbE, \bsbX \bsbDelta^B + \bsbDelta^C \rangle\mbox{.} \label{est:firstineq}
\end{align}
Here,  $\bsbDelta^B=\hat \bsbB -  \bsbB$, $\bsbDelta^C=\hat \bsbC -  \bsbC$ and so $r(\bsbDelta^B)\leq 2r$.
\begin{lemma} \label{lemma:phostochastic}
For any given $1\leq J \leq n, 1\leq r\leq m\wedge  p$,    define $\Gamma_{r, J} = \{(\bsbB, \bsbC)\in \mathbb R^{p\times m}\times \mathbb R^{n\times m} : r(\bsbB) \leq r, J(\bsbC) = J\}$. Then there exist universal constants $A_0, C, c>0$ such that for any $a\geq 2 b > 0$, the following event
\begin{align}
\sup_{(\bsbB, \bsbC)\in \bsbGamma_{r, J}} \Big\{2\langle \bsbE, \bsbX  \bsbB + \bsbC\rangle - \frac{1}{ a} \|\bsbX  \bsbB + \bsbC \|_{\tF}^2 - \frac{1}{b}    P_{2,H}(\bsbC; {\lambda}) - aA_0 \sigma^2 r(m+q) \Big\} \geq a  \sigma^2 t
\end{align}
occurs with probability at most  $ c' \exp(-c t)$, where $\lambda = A\lambda^o$, $\lambda^o=\sigma(m+ \log n)^{1/2}$, $A =  (a b A_1)^{1/2}$, $A_1\ge A_0$, and $t\geq 0$.
\end{lemma}

Let   $l_H(\bsbB, \bsbC, r) =  2\langle \bsbE, \bsbX  \bsbB + \bsbC\rangle -   \|\bsbX  \bsbB + \bsbC \|_{\tF}^2/a -       P_{2,H}(\bsbC; {\lambda})/b - aA_0 \sigma^2 r(m+q)$. Define  
$$
R=\sup_{1\leq J\leq n, 1\leq r \le m\wedge p}\, \ \sup_{(\bsbB, \bsbC)\in \bsbGamma_{r,J}} l_H(\bsbB, \bsbC, r)\mbox{.}
$$
From Lemma \ref{lemma:phostochastic},
it is easy to see $\EE R \leq a c\sigma^2$. Substituting the    bound below into \eqref{est:firstineq},
\begin{align*}
 2\langle \bsbE, \bsbX  \bsbDelta^B + \bsbDelta^C\rangle  \leq & \frac{1}{ a} \|\bsbX  \bsbDelta^B + \bsbDelta^C \|_{\tF}^2 +  \frac{1}{b}    P_{2,H}(\bsbDelta^C; {\lambda}) + 2aA_0 \sigma^2 r(m+q)+ R \\
\leq& \frac{2}{a} M(\bsbB-\bsbB^*, \bsbC - \bsbC^*) +  \frac{2}{a} M(\hat \bsbB-\bsbB^*, \hat\bsbC - \bsbC^*) \\ & +2 a A_0 \sigma^2 r(m+q) + R+   \frac{1}{b}    P_{2,H}(\bsbDelta^C; {\lambda}) ,
\end{align*}
 we have
\begin{align*}
(1-\frac{2}{a})M(\hat\bsbB-\bsbB^*, \hat\bsbC - \bsbC^*)
\leq & (1+\frac{2}{a})M(\bsbB-\bsbB^*, \bsbC - \bsbC^*) + 2aA_0 \sigma^2 r(m+q) + R \\
&+2  P(\bsbC; {\lambda}) - 2P(\hat \bsbC; {\lambda}) +  \frac{1}{b}    P_{2,H}(\bsbDelta^C; {\lambda})\mbox{.}
\end{align*}
It remains to deal with $2   P(\bsbC; {\lambda}) - 2P(\hat \bsbC; {\lambda})+    P_{2,H}(\bsbDelta^C; {\lambda})/b$ which is denoted by $I$ below.

\textit{(i) }Due to the sub-additivity of the  function $P_H$ that is concave on $[0,\infty)$,
\begin{align*}
I&\leq 2  P(\bsbC; {\lambda}) - 2P_{2,H}(\hat \bsbC; {\lambda})+ \frac{1}{b} P_{2,H}(\bsbDelta^C; {\lambda })\\
&\leq  2  P(\bsbC; {\lambda}) +  \frac{1}{b} P_{2,H}(\bsbC; {\lambda}) +   \frac{1}{b} P_{2,H}(\hat\bsbC; {\lambda }) - 2P_{2,H}(\hat \bsbC; {\lambda})\\
  &\le (2+\frac{1}{b})  P(\bsbC; {\lambda}),
\end{align*}
if $b\ge 1/2$.
Theorem 3 can be obtained by choosing $a = 4$, $b=1/2$,  and $\lambda = A \lambda^o$ with $A \ge (2 A_0)^{1/2}$. 

\textit{(ii)} When $P$ is the group $\ell_1$ penalty  as in Theorem 6,
   by the  sub-additivity of $P$, we have
   \begin{align*}
I&\leq 2P_{2,1}(\bsbC; {\lambda}) - 2P_{2,1}(\hat \bsbC; {\lambda})+ \frac{1}{b}   P_{2,1}(\bsbDelta^C; {\lambda}) \\
 &\leq  {2A\lambda^o}\{ (1 +\theta) \|\bsbDelta_{\mathcal J}^C\|_{2,1} - (1 - \theta) \|\bsbDelta_{{\mathcal J}^c}^C\|_{2,1}\}\\
&\leq  {2A(1 - \theta)\lambda^o}\{ (1 +\vartheta) \|\bsbDelta_{\mathcal J}^C\|_{2,1} -  \|\bsbDelta_{{\mathcal J}^c}^C\|_{2,1}\},
\end{align*}
where    $ \mathcal J(\bsbC)$ and   $J(\bsbC)$ are  abbreviated to $\mathcal J$, $J$, respectively, and we set    $b = 1/(2\theta)$, $\theta = \vartheta/(2+\vartheta)$. 
From the regularity condition,  $ (1 +\vartheta) \|\bsbDelta_{\mathcal J}^C\|_{2,1} -  \|\bsbDelta_{{\mathcal J}^c}^C\|_{2,1}\le K J^{1/2} \|(\bsbI - \Proj_{\bsbX \bsbDelta^B}) \bsbDelta^C\|_{\tF} \le K J^{1/2} \| {\bsbX \bsbDelta^B} + \bsbDelta^C\|_{\tF}$, and so
\begin{align*}
I & \le 2A(1-\theta) \lambda^o K J^{1/2} \| \bsbX \bsbDelta^B + \bsbDelta^C\|_{\tF}\\
&\le \frac{2}{a} M(\bsbB-\bsbB^*, \bsbC - \bsbC^*) +  \frac{2}{a} M(\hat \bsbB-\bsbB^*, \hat \bsbC - \bsbC^*) + a A^2 (1-\theta)^2 K^2 (\lambda^o)^2J \mbox{.}
\end{align*}
Taking $a = 4 + 1/\theta$, $b=1/(2\theta)$, and $A\ge (a b A_0)^{1/2}$ gives the conclusion in Theorem 6.

\subsection*{Proof of  Lemma \ref{lemma:phostochastic}}
\begin{proof}
Define
 $$
 l_H(\bsbB, \bsbC, r) =  2\langle \bsbE,  \bsbX  \bsbB + \bsbC\rangle - \frac{1}{ a} \|\bsbX  \bsbB + \bsbC \|_{\tF}^2 -  \frac{1}{b}   P_{2,H}(\bsbC; {\lambda }) - aA_0 \sigma^2 r(m+q)\mbox{.}
$$
 Similarly,   define  $l_0(\bsbB, \bsbC, r)$  with $P_{2,0}$  in place of     $P_{2,H}$ in the above. Let  $\mathcal A_H=\{\sup_{(\bsbB, \bsbC)\in \bsbGamma_{r, J}} \allowbreak l_H(\bsbB, \bsbC, r)\allowbreak\geq at\sigma^2\}$, and $\mathcal A_0=\{\sup_{(\bsbB, \bsbC)\in \bsbGamma_{r, J}} l_0(\bsbB, \bsbC, r)\geq at\sigma^2\}$.

Since   $\mathcal A_H\subset \{\sup_{(\bsbB, \bsbC):r(B) \leq r} l_H(\bsbB, \bsbC, r) \geq a t\sigma^2\}$, the occurrence of  $\mathcal A_H$ implies that 
\begin{align}
 l_H(\bsbB^o, \bsbC^o, r) \geq a t \sigma^2,\label{auxbnd}
 \end{align}
  for any  $(\bsbB^o ,\bsbC^o) $ that solves
  \begin{align}
 \min_{\bsbB: r(\bsbB)\leq r, \bsbC}  \frac{1}{ a } \|\bsbX  \bsbB + \bsbC \|_{\tF}^2 -2\langle \bsbE, \bsbX  \bsbB + \bsbC\rangle +  \frac{1}{b}   P_{2,H}(\bsbC; {\lambda }) \mbox{.} \label{auxopt}
  \end{align}

  \begin{lemma} \label{lemma:phcomp}
Given any $\theta\ge 1$, there exists a globally optimal solution $\bsbC^o$ to
$
\min_{\bsbC}  \| \bsbY - \bsbC\|_{\tF}^2 /2 + \theta P_{2,H}(\bsbC; {\lambda})
$
such that for any $i: 1\leq i \leq n$, either $\bsb{c}_i^o=\bsb{0}$ or $\|\bsb{c}_i^o\|_2 \geq \lambda  {\theta}^{1/2}\ge \lambda$.
\end{lemma}

See \cite{She2012} for its proof. From Lemma \ref{lemma:phcomp} and $a \geq 2 b$, \eqref{auxbnd} further indicates that there exists an optimal solution $(\bsbB^o, \bsbC^o)$ such that $l_0(\bsbB^o,\bsbC^o, r) \geq a t \sigma^2$. Hence  $\mathcal A_H \subset \mathcal A_0$ and it suffices to show $\EP (\mathcal A_0) \le C \exp(-ct)$. 

Let $\mathcal J = \mathcal J (\bsbC)$ for short. Denote by $\bsbI_{\mathcal J}$ the submatrix of $\bsbI_{n\times n}$ formed by the columns indexed by $\mathcal J$. We write the stochastic term into
\begin{align}
2\langle \bsbE, \bsbX \bsbB+\bsbC\rangle \notag
= \, &2\langle \bsbE, \Proj_{\bsbI_{{\mathcal J}}}^{\perp} \bsbX \bsbB\rangle+ 2\langle \bsbE, \Proj_{\bsbI_{{\mathcal J}}}  ( \bsbX \bsbB + \bsbC)  \rangle \notag\\
 \equiv \,& 2\langle \bsbE, \bsbA_1  \rangle + 2\langle \bsbE, \bsbA_2\rangle, \label{finedecomp}
\end{align}
and $\| \bsbA_1\|_{\tF}^2 + \|\bsbA_2\|_{\tF}^2 = \|\bsbX \bsbB + \bsbC\|_{\tF}^2$.

\begin{lemma}\label{concenGauss}
Given $\bsbX\in \mathbb R^{n\times p}$, $1\leq J\leq n$, $1\leq r \leq m\wedge p$,  define $\Gamma_{r, J}^1 = \{\bsbA\in \mathbb R^{n\times m}: \|\bsbA\|_{\tF}\leq 1, r(\bsbA) \leq r, CS(\bsbA) \subset CS\{\bsbX({\mathcal J}^c, :)\} \mbox{ for some } \mathcal J: | \mathcal J|=J\}$. Let $$P_o^1(J, r) = \sigma^2 \left[\{q\wedge  (n-J) \} r+(m - r) r + \log {n\choose J}\right]\mbox{.}$$
Then for any $t\geq 0$,
\begin{align}
\EP \Big[\sup_{\bsbA \in \Gamma_{r,J}^1} \langle \bsbE, \bsbA \rangle \geq t \sigma +   \{L   P_o^1(J,r)\}^{1/2}\Big] \leq c'\exp(- ct^2),
\end{align}
where $L,    c, c'>0$ are universal constants.
\end{lemma}

The proof follows similar lines of the proof of Lemma 4 in \cite{She2016b} and is omitted.
Now,
we can bound the   the first term on the right hand side of \eqref{finedecomp} as follows

\begin{align*}
&2\langle \bsbE, \bsbA_1  \rangle - \frac{1}{a}\| \bsbA_1\|_{\tF}^2 - 2a L P_o^1(J,r) \\
\le & 2\langle \bsbE, \bsbA_1/\|\bsbA_1\|_{\tF}  \rangle \|\bsbA_1\|_{\tF}  -2 \| \bsbA_1\|_{\tF} \{ LP_o^1(J,r)\}^{1/2}  - \frac{1}{2a}\| \bsbA_1\|_{\tF}^2 \\
 \le & 2a \left[\langle \bsbE, \bsbA_1/\|\bsbA_1\|_{\tF}  \rangle   - \{  LP_o^1(J,r)\}^{1/2}\right ]_+^2 + \frac{1}{2a}\| \bsbA_1\|_{\tF}^2- \frac{1}{2a}\| \bsbA_1\|_{\tF}^2\\
=&2a  \left[\langle \bsbE, \bsbA_1/\|\bsbA_1\|_{\tF}  \rangle   - \{  LP_o^1(J,r)\}^{1/2}\right ]_+^2\mbox{.}
\end{align*}
By Lemma \ref{concenGauss}, for $L$ large enough,
$$
\EP \{2\langle \bsbE, \bsbA_1  \rangle - \frac{1}{a}\| \bsbA_1\|_{\tF}^2 - 2a L P_o^1(J,r) > \frac{1}{2}a t\sigma^2\}\ \leq c' \exp(-c t)\mbox{.}
$$
Similarly,  for the second term on the right hand side of \eqref{finedecomp}, $$
\EP \{2\langle \bsbE, \bsbA_2  \rangle - \frac{1}{a}\| \bsbA_2\|_{\tF}^2 - 2a L P_o^2(J,r) > \frac{1}{2}a t\sigma^2\}\ \leq c' \exp(-c t),
$$
 where $$P_o^2 (J, r) = \sigma^2\left\{J m + \log {n \choose J}\right\},$$ and $L$ is a large constant. Applying  the union bound gives
\begin{align}
 & \EP [2\langle \bsbE, \bsbX \bsbB + \bsbC\rangle - \frac{1}{a}\|\bsbX \bsbB + \bsbC\|_{\tF}^2 - 2a L \sigma^2\{ (q+m-r)r + Jm + J\log(en/J) \}>  a t\sigma^2]\ \notag \\ \leq \ &  c' \exp(-c t).\label{stobndinlemma}
\end{align}
The conclusion follows.
\end{proof}

\subsection{Proof of Theorem 4}
 Similar to   Section \ref{proof_oraclel0l1}, we have
\begin{align*}
\frac{1}{2}M( \hat \bsbB -  \bsbB^*, \hat \bsbC -  \bsbC^*)\leq \frac{1}{2}M( \bsbB -  \bsbB^*, \hat \bsbC -  \bsbC^*)  + \langle \bsbE, \bsbX \bsbDelta^B + \bsbDelta^C \rangle,  
\end{align*}
where   $\bsbDelta^B=\hat \bsbB -  \bsbB$, $\bsbDelta^C=\hat \bsbC -  \bsbC$. Let $\tilde r = r( \bsbDelta^B)$ and $\tilde J =  J(\bsbDelta^C)$.
Then from \eqref{stobndinlemma} in the proof of Lemma \ref{lemma:phostochastic}, $$
2\langle \bsbE, \bsbX \bsbDelta^B + \bsbDelta^C\rangle \le  \frac{1}{a}\|\bsbX \bsbDelta^B + \bsbDelta^C\|_{\tF}^2 - 2a L \sigma^2\{ (q+m)\tilde r + \tilde J m + \tilde J\log(en/\tilde J)\}\ + R,
$$
where $ER\le ac \sigma^2$. The oracle inequality can be shown following the lines of Section \ref{proof_oraclel0l1},   noticing that   $\tilde r \leq 2r$, $\tilde J \le 2\outl$ and $\tilde J\log(2en/\tilde J) \le 2 \outl \log (en/\outl)$.

\subsection{Proof of Theorem 5}
The proof is based on the general reduction scheme  in Chapter 2 of \cite{tsybakov2009}.
We consider two cases.

\textit{Case (i)} $(q +m ) r \geq Jm+  J\log (e n / J)$. Suppose the SVD of $\bsbX$ is    $\bsbX = \bsbU \bsbD \bsbV^\T$ with $\bsbD$ of size $q\times q$.  Given  an arbitrary  estimator $(\hat \bsbB, \hat \bsbC)$, let  $\hat \bsbA= \bsbV^\T \hat\bsbB$ and   $\tilde {\mathcal S}(r, J) =\{(\bsbA, \bsbC) \in \mathbb R^{q\times m}\times \mathbb R^{n\times m}: r(\bsbA)\leq r, J(\bsbC)\leq J\}$. Then
 \begin{align*}
&\sup_{(\bsbB^*, \bsbC^*) \in \mathcal S(r, J)} \EP\{\|\bsbX \bsbB^* - \bsbX \hat\bsbB + \bsbC^* - \hat\bsbC\|_{\tF}^2\geq c {P_o(J, r)}\}\\
\geq & \sup_{ (\bsbA^*, \bsbC^*) \in  \tilde{\mathcal S}(r, J)} \EP\{\|\bsbU \bsbD\bsbA^* - \bsbU\bsbD \hat\bsbA + \bsbC^* - \hat\bsbC\|_{\tF}^2\geq c {P_o(J, r)}\},
\end{align*} because  for any $\bsbA:r(\bsbA)\leq r$,  $ \bsbB = \bsbV \bsbA$ satisfies $r(\bsbB)\leq r$. The new design matrix   $ \bsbU\bsbD$    has $q$ columns, and it is easy to see that  for any $\bsbA\in\mathbb R^{q\times m}$, \begin{align}\underline{\kappa} \|\bsbA\|_{\tF}^2 \le \| \bsbU \bsbD \bsbA\|_{\tF}^2\le \overline{\kappa} \|\bsbA\|_{\tF}^2, \label{condinumber}
\end{align} where $\underline{\kappa} = \sigma_{\min}^2(\bsbX)$ and $\overline{\kappa} = \sigma_{\max}^2(\bsbX)$ as defined in the theorem.  Therefore, without any loss of generality we assume $\bsbX \in \mathbb R^{n\times q}$ and  and    $\bsbB\in \mathbb R^{q\times m}$  in the rest of the proof.

Consider a signal  subclass \begin{align*}
{\mathcal B}^1(r)=\{\bsbB = (b_{jk}),  \bsbC=\bsb{0}: & \ b_{jk}\in \{0, \gamma R\} \mbox{ if }  (j,k)\in[q]\times[ r/2]\cup [r/2]\times[ m] \\ & \mbox{   }   b_{jk}=0 \mbox{ otherwise} \}.
\end{align*}
where $R=  {\sigma}/({{\overline{\kappa}^{1/2}} })$,  and $\gamma>0$ is a small constant to be chosen later.
Clearly,  $|{\mathcal B}^1 (r)|= 2^{(q+m-r/2)r/2}$, $\mathcal B^1 (r)\subset \mathcal S(r, J)$, and $r(\bsbB_1 - \bsbB_2)\leq r$, for any $  \bsbB_1, \bsbB_2 \in \mathcal B^1(r)$. Also, since $r\le q\wedge m$, $(q+m-r/2)r/2\geq c(q+m)r$ for some   constant $c$.

Let $\rho(\bsbB_1, \bsbB_2)=\|\vect(\bsbB_1) - \vect(\bsbB_2)\|_0$,  the Hamming distance between $\vect(\bsbB_1) $ and $\vect(\bsbB_2)$. By the  Varshamov-Gilbert bound,  cf. Lemma 2.9 in \cite{tsybakov2009},
there exists  a subset ${\mathcal B}^{10}(r)\subset {\mathcal B}^{1}(r)$ such that
\begin{eqnarray*}
\log | {\mathcal B}^{10}(r)| \geq c_1 r(q+m), \quad
\rho(\bsbB_1, \bsbB_2) \geq c_2 r(q+m), \bsbB_1, \bsbB_{2} \in \mathcal B^{10},  \bsbB_1\neq \bsbB_2
\end{eqnarray*}
for   some universal constants $c_1, c_2>0$.
Then $\| \bsbB_1 - \bsbB_2\|_{\tF}^2 = \gamma^2 R^2 \rho(\bsbB_1, \bsbB_2) \geq c_2 \gamma^2 R^2 (q+m) r$. It follows from \eqref{condinumber} that
\begin{align}
\| \bsbX \bsbB_1 - \bsbX \bsbB_2  \|_{\tF}^2 \geq c_2 \underline{\kappa}   \gamma^2 R^2 (q+m)r \label{separationLBound}
\end{align}
 for any $\bsbB_1, \bsbB_{2} \in \mathcal B^{10}$,  $\bsbB_1\neq \bsbB_2$, where $\underline{\kappa}/\overline{\kappa}$ is a positive constant.

For Gaussian models, the  Kullback-Leibler divergence of $\mathcal {M N}( \bsbX \bsbB_2,\sigma^2 \bsb{I}\otimes\bsbI)$,   denoted by $P_{\bsbB_2}$, from $\mathcal {MN}( \bsbX \bsbB_1),\sigma^2 \bsb{I}\otimes \bsbI)$,  denoted by $P_{\bsbB_1}$, is $$\mathcal K(\mathcal P_{\bsbB_1}, \mathcal P_{\bsbB_2}) = \frac{1}{2\sigma^2} \|\bsbX \bsbB_1 -  \bsbX \bsbB_2 \|_{\tF}^2\mbox{.}$$ Let $P_{\bsb{0}}$ be $\mathcal {M N}(\bsb{0}, \sigma^2 \bsb{I}\otimes \bsbI)$. By \eqref{condinumber}  again,
for any $\bsbB: r(\bsbB)\leq r$, we have
\begin{align*}
\mathcal K(P_{\bsb{0}}, P_{\bsbB}) \leq \frac{1}{2\sigma^2}\overline{\kappa}  \gamma^2 R^2 \rho(\bsb{0}, \bsbB) \leq \frac{\gamma^2}{\sigma^2}\overline{\kappa}   R^2 (q+m)r,
\end{align*}
where we used $\rho(\bsbB_1, \bsbB_2) \leq r(q+m)$. Therefore,
\begin{align}
\frac{1}{|\mathcal B^{10}|}\sum_{\bsbB\in \mathcal B^{10}} \mathcal K(P_{\bsb{0}}, P_{\bsbB})\leq \gamma^2 r(q+m)\mbox{.} \label{KLUBound}
\end{align}

Combining  \eqref{separationLBound} and \eqref{KLUBound} and choosing a sufficiently small value for  $\gamma$, we can apply Theorem 2.7 of \cite{tsybakov2009} to get the desired  lower bound.

\textit{Case (ii)} $(q+m)r < Jm+J \log (e n /J) $.
Define a signal subclass
\begin{align*}
{\mathcal B}^2(J)=&\{\bsbB , \bsbC=(\bsbc_{1}, \ldots, \bsbc_n)^\T:   \bsbB= \bsb{0}, \bsbc_i = \bsb{0}  \mbox{ or }  \gamma R   (\bsb{1}^\T, \bsb{b}^\T)^\T \\&\mbox{ with } \bsb1=(1 \ \ldots \ 1)^\T\in \mathbb R^{m -\lceil m/2\rceil },   \bsb{b}\in \mathbb \{0,1\}^{\lceil m/2\rceil},   J(\bsbC)\leq J \}\mbox{.}
\end{align*}
where $$R= \frac{\sigma}{{\overline{\kappa}^{1/2}}  } \left\{1+\frac{\log (e n /J)}{m}\right\}^{1/2}, $$ and $\gamma>0$ is a small constant.
 Clearly,  ${\mathcal B}^{2}(J) \subset\mathcal{S}(r,J) $.
By Stirling's approximation, $$\log |\mathcal B^{2}(J)|\geq \log { n \choose J} +  \log 2 ^{J m/2}\geq   J \log (n/J) + J m (\log 2)/2 \geq c \{J \log ( en/J)+Jm\}$$ for some universal constant $c$.
Applying    Lemma 8.3 in   \cite{Rigollet11} and the Varshamov-Gilbert  bound,  there exists  a subset ${\mathcal B}^{20}(J)\subset {\mathcal B}^{2}(J)$ such that
\begin{eqnarray*}
\log | {\mathcal B}^{20}(J)| \geq c_1 \{J \log ( e n/J) + J m\}
\mbox{ and }
\rho(\bsbB_1, \bsbB_2) \geq c_2 Jm, \forall \bsbB_1, \bsbB_{2} \in \mathcal B^{20},  \bsbB_1\neq \bsbB_2
\end{eqnarray*}
for some universal constants $c_1, c_2>0$.
The afterward treatment follows the same lines as in (i) and the details are omitted.

\subsection{Proof of Theorem 7}

The first conclusion follows  from the block coordinate descent   design and     the optimality of the multivariate thresholding for solving the $\bsbC$-optimization problem \citep{She2012}.

When the continuity condition holds, $\vec \Theta(\bsbY - \bsbX \bsbB; \lambda)$ is the unique minimizer of    $ \min_{\bsbC} F(\bsbB, \bsbC)  $; see Lemma 1  of \cite{She2012}. But  in general,  the problem of    $\min_{\bsbB} F(\bsbB, \bsbC)$ subject  to $r(\bsbB)\le r$ may not have a   unique solution. The accumulation point result is an application of Zangwill's Global Convergence Theorem \citep{luenberger08}, and the proof    proceeds along similar   lines of the proof of Theorem 7 of  \cite{bunea2012}.
 The details are omitted.

To get the stationarity guarantee when $q(\cdot;\lambda)\equiv 0$, we can write the problem as $\min  \| \bsbY - \bsbX \bsbS \bsbV^\T - \bsbC\|_{\tF}^2/2 + \sum_{i=1}^n P_{\Theta} (\| \bsbc_i\|_2;\lambda)$ subject to $(\bsbS, \bsbV, \bsbC) \in \mathbb R^{p\times r} \times \mathbb O^{m\times r} \times \mathbb R^{n\times m}$, where $\mathbb O^{m\times r}=\{\bsbV\in \mathbb R^{m\times r}: \bsbV^\T \bsbV = \bsbI\}$. Then one can view the problem as an unconstrained one on the    manifold $\mathbb R^{p\times r} \times \mathbb O^{m\times r} \times \mathbb R^{n\times m}$, and  define  the Remannian gradient with respect to $\bsbV$;  see Theorem 6 of \cite{bunea2012} for more detail.

\subsection{Proof of Theorem 8}

First,    by a bit of algebra we have the following result.
\begin{lemma}\label{A6}
 For any $(\hat \bsbB, \hat \bsbC)$ defined in the theorem, we have  $$(\hat \bsbB, \hat \bsbC) \in \arg \min_{(\bsbB, \bsbC)  } g(\bsbB, \bsbC; \bsbB^- , \bsbC^-)|_{\bsbB^- = \hat \bsbB, \bsbC^-=\hat \bsbC} \ \mbox{ s.t. } r(\bsbB)\le r,$$ where  $g$ is constructed by $g(\bsbB, \bsbC; \bsbB^- , \bsbC^-) =  l(\bsbB^-, \bsbC^-) + P_{2,\Theta}(\bsbC;\lambda)+ \langle \bsbX \bsbB^- + \bsbC^- -\bsbY, \bsbX \bsbB - \bsbX \bsbB^- + \bsbC   - \bsbC^-\rangle \allowbreak + \allowbreak\| \bsbX \bsbB - \bsbX \bsbB^-\|_{\tF}^2/2+ \|  \bsbC -  \bsbC^-\|_{\tF}^2/2$, with  $l(\bsbB, \bsbC) = \|\bsbX \bsbB + \bsbC - \bsbY\|_{\tF}^2/2$ and $P_{2,\Theta}(\bsbC;\lambda)=\sum_{i=1}^n P_{\Theta}(\|\bsbc_i\|_2;\lambda)$.
\end{lemma}

The following result can be obtained from   Lemma 2 in \cite{She2012}.
\begin{lemma}\label{A7}
 Let $Q (\bsbC)= \|\bsbC - \bsbY\|_{\tF}^2/2 + P_{2,\Theta}(\bsbC;\lambda)$ and $\bsbC^o=\vec\Theta(\bsbY;\lambda)$. Assume that $\vec\Theta$ is continous at $\bsbY$.  Then   for any $\bsbC$,
$
Q(\bsbC )-Q (\bsbC^o) \geq {( 1-{\mathcal L}_{\Theta})} \|\bsbC - \bsbC^o \|_{\tF}^2/2
$.
\end{lemma}

\begin{lemma}\label{A8}
 Let $Q (\bsbB)= \| \bsbX\bsbB - \bsbY\|_{\tF}^2/2 $  and $\bsbB^o= \mathcal R(   \bsbX, \bsbY, r)$ which is of rank $r$. Then for any  $\bsbB: r(\bsbB)\le r/(1+\alpha)$ with $\alpha\ge 0$,
$
Q(\bsbB)-Q (\bsbB^o) \geq \{  1- (1+\alpha)^{-1/2}\} \|\bsbX\bsbB - \bsbX \bsbB^o\|_{\tF}^2/2
$.
\end{lemma}

The lemma follows from Proposition 2.2 of \cite{She2013} and Lemma \ref{le:basic_l0} below.
\begin{lemma} \label{le:basic_l0}
The optimization problem $\min_{\bsbb\in \mathbb R^p} l(\bsbb)= \|\bsby - \bsbb\|_2^2/2 \mbox{  s.t. } \|\bsbb\|_0\le q$ has    $\hat\bsbb = \Theta^{\#}(\bsby; q)$ as a globally optimal solution.    Assume that   $J(\hat\bsbb)=q$, where $J(\cdot) = \| \cdot\|_0$. Then   for any $\bsbb$ with   $ J(\bsbb)\le s =  q/\theta $ and $\theta \ge 1$, we have
$
l (\bsbb) - l(\hat\bsbb) \ge \{1-\mathcal L(\mathcal J, \hat{\mathcal  J})\}\| \hat\bsbb - \bsbb\|_2^2/2
$
where   $ \mathcal L(\mathcal J, \hat{\mathcal  J}) = ({| \mathcal J \setminus \hat{\mathcal  J}| }/{| \hat{\mathcal  J} \setminus \mathcal J|})^{1/2}\le ({s}/{q})^{1/2}=  {\theta}^{-1/2}$, $\mathcal J = \mathcal J(\bsbb)$ and $\hat{\mathcal J} = \mathcal J(\hat \bsbb)$.
\end{lemma}

With Lemmas \ref{A6}, \ref{A7}, and \ref{A8} available, the conclusion results from Theorem 2 of \cite{She2016a}.

\subsection*{Proof of Lemma \ref{le:basic_l0}}
\begin{proof}
Let $\mathcal J_1 = \mathcal J \cap \hat{\mathcal J}$, $\mathcal J_2 =  \hat{\mathcal J}\setminus \mathcal J$ and $\mathcal J_3 = \mathcal J \setminus \hat{\mathcal J}$. Then   $\bsbb=\bsbb_{\mathcal J_1}+\bsbb_{\mathcal J_3}$ and $\hat\bsbb=\bsbb_{\mathcal J_1}+\bsbb_{\mathcal J_2}$. By writing  $\bsbb_{\mathcal J_1} = \bsby_{{\mathcal J}_1}+\bsbdelta_{{\mathcal J}_1}$ and  $\bsbb_{\mathcal J_3} =  \bsby_{{\mathcal J}_3}+\bsbdelta_{{\mathcal J}_3}$, we have
\begin{align*}
l (\bsbb ) - l (\hat\bsbb )& = \frac{1}{2} \|\bsbdelta_{{\mathcal J}_1}\|_2^2 + \frac{1}{2} \|\bsby_{{\mathcal J}_2}\|_2^2+\frac{1}{2} \|\bsbdelta_{{\mathcal J}_3}\|_2^2- \frac{1}{2} \|\bsby_{{\mathcal J}_3}\|_2^2\\
\frac{1}{2}\| \hat\bsbb - \bsbb\|_2^2&=\frac{1}{2} \|\bsbdelta_{{\mathcal J}_1}\|_2^2+\frac{1}{2} \|\bsby_{{\mathcal J}_2}\|_2^2+ \frac{1}{2} \|\bsby_{{\mathcal J}_3}+\bsbdelta_{{\mathcal J}_3}\|_2^2\mbox{.}
\end{align*}
The key lies in the comparison between     $   \|\bsby_{{\mathcal J}_2}\|_2^2+  \|\bsbdelta_{{\mathcal J}_3}\|_2^2-   \|\bsby_{{\mathcal J}_3}\|_2^2$ and $   \|\bsby_{{\mathcal J}_2}\|_2^2+ \|\bsby_{{\mathcal J}_3}+\bsbdelta_{{\mathcal J}_3}\|_2^2$.
Let  $K\le 1$ satisfy
$$
\frac{1}{2} \|\bsby_{{\mathcal J}_2}\|_2^2+\frac{1}{2} \|\bsbdelta_{{\mathcal J}_3}\|_2^2- \frac{1}{2} \|\bsby_{{\mathcal J}_3}\|_2^2 \ge  \frac{K}{2} \|\bsby_{{\mathcal J}_2}\|_2^2+ \frac{K}{2} \|\bsby_{{\mathcal J}_3}+\bsbdelta_{{\mathcal J}_3}\|_2^2,
$$
which is equivalent
to \begin{align}
(1-K)\|\bsby_{{\mathcal J}_2}\|_2^2+ \|\bsbdelta_{{\mathcal J}_3}\|_2^2   \ge K  \|\bsby_{{\mathcal J}_3}+\bsbdelta_{{\mathcal J}_3}\|_2^2 + \|\bsby_{{\mathcal J}_3}\|_2^2\mbox{.}\label{intermedineql0}
\end{align}
By construction,   $ \lvert y_i \rvert \ge \lvert y_j \rvert$ for any $ i \in {\mathcal J}_2$ and $ j \in {\mathcal J}_3$. Thus   $\|\bsby_{{\mathcal J}_2}\|_2^2/J_2 \ge \|\bsby_{{\mathcal J}_3}\|_2^2/J_3$, from which it follows that  \eqref{intermedineql0} is implied by 
\begin{align*}
(1-K)\frac{J_2}{J_3}\|\bsby_{{\mathcal J}_3}\|_2^2+ \|\bsbdelta_{{\mathcal J}_3}\|_2^2 \ge (1+K) \|\bsby_{{\mathcal J}_3}\|_2^2+ K \| \bsbdelta_{{\mathcal J}_3}\|_2^2 + 2K \langle \bsby_{{\mathcal J}_3}, \bsbdelta_{{\mathcal J}_3}\rangle, \end{align*}
or
\begin{align*}
\frac{(1-K) ({J_2}/{J_3})-(1+K)}{K}\|\bsby_{{\mathcal J}_3}\|_2^2+ \frac{1-K}{K} \|\bsbdelta_{{\mathcal J}_3}\|_2^2 \ge    2 \langle \bsby_{{\mathcal J}_3}, \bsbdelta_{{\mathcal J}_3}\rangle\mbox{.} \end{align*}
Therefore, the largest possible $K$ satisfies
$$
\frac{(1-K)({J_2}/{J_3})-(1+K)}{K} \times \frac{1-K}{K}=1
$$
or
$
(1-K)^2 = J_3/  J_2$. This gives
$$
\mathcal L = 1-K =   ({ {J_3}/{J_2}} )^{1/2} \le \{({J_3+J_1})/({J_2+J_1})\}^{1/2}= ({ {J}/{\hat  J}})^{1/2} \le  {\theta}^{-1/2}\mbox{.}
$$
The proof is complete.
\end{proof}

\subsection{Proof of Theorem 9}

  Let $h(\bsbB,  \bsbC; A) = 1/\{mn - A P  (\bsbB, \bsbC)\}$.
It follows from  $1/(1-\delta) \ge \exp(\delta)  $ for any $0\le \delta < 1$ and $\exp(\delta) \ge 1/(1-\delta/2)$ for any $0\ge\delta<2$ that   \begin{align*}
 mn\| \bsbY - \bsbX \hat \bsbB - \hat \bsbC\|_{\tF}^2 \,  h(\hat \bsbB, \hat \bsbC; A/2)   
  \le & {\| \bsbY - \bsbX \hat \bsbB-\hat \bsbC\|_{\tF}^2}\exp\{ \delta(\hat \bsbB, \hat \bsbC)\} \\ \le & {\| \bsbY - \bsbX   \bsbB^*-  \bsbC^*\|_{\tF}^2}\exp\{ \delta( \bsbB^*,  \bsbC^*)\} \\
\le & \| \bsbY - \bsbX   \bsbB^*- \bsbC^*\|_{\tF}^2 \, h( \bsbB^*,  \bsbC^*;  A) mn\mbox{.}
\end{align*}
Since $ h(\hat \bsbB, \hat \bsbC; A/2)> 0$, we have     $$\| \bsbY - \bsbX \hat \bsbB - \hat \bsbC\|_{\tF}^2\le \| \bsbY - \bsbX   \bsbB^* - \bsbC^*\|_{\tF}^2 \, h(  \bsbB^*, \bsbC^*;  A)/ h(\hat \bsbB, \hat \bsbC; A/2)\mbox{.}$$    With a bit of algebra, we get
\begin{align*}
 M( \hat \bsbB -  \bsbB^*, \hat \bsbC -  \bsbC^*) 
\leq & \|\bsbE\|_{\tF}^2 \{{h(\bsbB^*, \bsbC^*;  A)}/{h(\hat\bsbB, \hat\bsbC;  0.5A)}-1\} \\ & + 2\langle \bsbE, \bsbX \hat \bsbB - \bsbX \bsbB^* + \hat\bsbC - \bsbC^*\rangle \\
\leq & \frac{ A  \|\bsbE\|_{\tF}^2 }{mn\sigma^2 -  A \sigma^{2} P(\bsbB^*, \bsbC^*) }\sigma^2 P(\bsbB^*, \bsbC^*) - \frac{0.5 A  \|\bsbE\|_{\tF}^2}{mn \sigma^2}\sigma^2P(\hat\bsbB, \hat \bsbC) \\& + 2\langle \bsbE, \bsbX \hat \bsbB - \bsbX \bsbB^*+ \hat\bsbC - \bsbC^*\rangle\mbox{.}
\end{align*}

We give a finer treatment of the last stochastic term than that  in the proof of  Lemma \ref{lemma:phostochastic}, to show that $\langle \bsbE, \bsbX \hat \bsbB - \bsbX \bsbB^*+ \hat\bsbC - \bsbC^*\rangle$ can be bounded by   $ P(\bsbB^*, \bsbC^*)  +  P(\hat \bsbB, \hat \bsbC)   $ up to a multiplicative constant with high probability.
Let $\bsbDelta^B = \hat \bsbB - \bsbB^*$, $\bsbDelta^C = \hat \bsbC - \bsbC^*$, $\hat {\mathcal J} = \mathcal J ( \hat \bsbC)$,  $  {\mathcal J}^* = \mathcal J (  \bsbC^*)$, $\hat r = r ( \hat \bsbB)$, $r^* = r (   \bsbC^*)$. In the following, given any index set $\mathcal J\subset [n]$, we denote by $\bsbI_{\mathcal J}$ the submatrix of $\bsbI_{n\times n}$ formed by the columns indexed by $\mathcal J$, and   abbreviate $\Proj_{\bsbI_\mathcal J}$ to  $\Proj_{\mathcal J}$. Let $\Proj_1 = \Proj_{\mathcal J^*}$,  $\Proj_2 = \Proj_{(\mathcal J^{*  })^c \cap \hat {\mathcal J}}$,  $\Proj_3 = \Proj_{(\mathcal J^{*  }  \cup \hat {\mathcal J})^c}$, and $\Proj_{rs}$ be the orthogonal projection onto the row space of $\bsbX \bsbB^*$ which is of rank $\le r^*$.
Then
\begin{align}
&  \bsbX \bsbDelta^B  - \bsbDelta^C  \notag\\
= \, &    \Proj_{1} ( \bsbX \bsbDelta^B  - \bsbDelta^C) +  \Proj_{2} ( \bsbX \bsbDelta^B  - \bsbDelta^C)  +  \Proj_{3} ( \bsbX \bsbDelta^B  - \bsbDelta^C) \Proj_{rs} +    \Proj_{3} ( \bsbX \bsbDelta^B  - \bsbDelta^C) \Proj_{rs}^\perp \notag\\
 \equiv \,&    \bsbDelta_1     +   \bsbDelta_2 + \bsbDelta_3 + \bsbDelta_4 , \notag
\end{align}
and $\sum_{i=1}^4\| \bsbDelta_i\|_{\tF}^2 = \|\bsbX \bsbDelta^B  - \bsbDelta^C\|_{\tF}^2$.
 Then $CS(\bsbDelta_1)\subset \Proj_{{\mathcal J}^*}$, $CS(\bsbDelta_2)\subset \Proj_{\hat {\mathcal J}}$, $r(\bsbDelta_3)\le r^*$,  and $r(\bsbDelta_4) = r(\Proj_{3}   \bsbX \bsbDelta^B    \Proj_{rs}^\perp)=r(\Proj_{3}   \bsbX \hat\bsbB     \Proj_{rs}^\perp) \le \hat r$.  The stochastic term can then be handled in a way similar to   that in   Lemma \ref{lemma:phostochastic}. For example, we can use the following result to handle $\langle \bsbE, \bsbDelta_4\rangle$.

\begin{lemma}
Given $\bsbX\in \mathbb R^{n\times p}$, $1\leq J_1, J_2\leq n$, $1\leq r \leq m\wedge p$, define $\Gamma_{r, J_1, J_2}  = \{\bsbA\in \mathbb R^{n\times m}: \|\bsbA\|_{\tF}\leq 1, r(\bsbA) \leq r, CS(\bsbA) \subset CS[\bsbX\{({\mathcal J}_1\cup {\mathcal J}_2)^c, :\}] \mbox{ for some } {\mathcal J}_1, {\mathcal J}_2: | \mathcal J_1|=J_1,    | \mathcal J_2|=J_2\}$. Let $$P_o (J_{1}, J_{2} , r) = \sigma^2 \left\{ q   r+(m - r) r + \log {n\choose J_1}+\log {n\choose J_2}\right\}\mbox{.}$$
Then for any $t\geq 0$,
\begin{align}
\EP \Big[\sup_{\bsbA \in \Gamma_{r,J_1, J_2}} \langle \bsbE, \bsbA \rangle \geq t \sigma +   \{L   P_o (J_{1}, J_2,r)\}^{1/2}\Big] \leq c'\exp(- ct^2),
\end{align}
where $L,   c, c'>0$ are universal constants.
\end{lemma}

Following the lines of   the proof of Theorem 2 in \cite{She2016b}, we can show that for any constants $a, b, a'>0$ satisfying $4b >a$, the following event
$$
 2\langle \bsbE, \bsbX \bsbDelta^B  - \bsbDelta^C \rangle
\leq2  ({1}/{a}+{1}/{a'})M( \hat \bsbB -  \bsbB^*, \hat \bsbC -  \bsbC^*)   + 8b L \sigma^2\{P (\hat\bsbB ,\hat\bsbC) + P(\bsbB^*, \bsbC^*)\}
$$
occurs with probability at least $1-c_1' n^{-c_1}$ for some $c_1, c_1'>0$, where $L$ is a sufficiently large constant.

Let $\gamma$ and $\gamma'$ be  constants satisfying $0< \gamma <\ 1, \gamma'>0$. On $\mathcal A =\{(1-\gamma) {mn\sigma^2}\leq \|\bsbE\|_{\tF}^2 \leq (1+\gamma') {mn\sigma^2} \}$ ,  we have
\begin{align*}
&\frac{ A  \|\bsbE\|_{\tF}^2}{mn\sigma^2 -  A \sigma^{2} P(\bsbB^*, \bsbC^*)} \sigma^2P(\bsbB^*, \bsbC^*) - \frac{0.5A  \|\bsbE\|_{\tF}^2}{mn \sigma^2}\sigma^2P (\hat\bsbB, \hat \bsbC) \\ \leq \ &  \frac{(1+\gamma') A A_0}{A_0 -  A}\sigma^2P (\bsbB^*, \bsbC^*)-0.5   {(1-\gamma) A\sigma^2 }P(\hat\bsbB, \hat \bsbC)\mbox{.}
\end{align*}
From \cite{Laurent2000}, the complement of $\mathcal A$ occurs with probability at most $c_2'\exp(-c_2 m n)$,   where  $c_2, c_2'$ are dependent   on constants $\gamma, \gamma'$.
 With $A_0$   large enough, we can choose  $a, a', b, A$ such that    $({1}/{a}+{1}/{a'})<{1}/{2}$,  $4  b>a$, and $16b L\leq ( 1-\gamma)A$. The conclusion results.


\subsection{ Theorem 10}

\setcounter{theorem}{9}

\begin{theorem}
\label{th_est}
Let $(\hat \bsbB, \hat\bsbC) = \arg \min_{(\bsbB, \bsbC)  } \| \bsbY - \bsbX \bsbB-\bsbC\|_{\tF}^2/2 + \lambda \|\bsbC\|_{2,1}$ subject to $r(\bsbB)\leq r$,    $\lambda =A\sigma (m+\log n)^{1/2}$ where  $r\ge r^* \ge 1$ and  $A$ is a  large enough constant. Assume that      $\bsbX$ satisfies  $
   (1+\vartheta) \lambda \| \bsbC'_{\mathcal J^*}\|_{2,1} + n\| \bsbB'\|_{\tF}^2   \le   \lambda \| \bsbC'_{{{\mathcal J}^*}^c}\|_{2,1}  +\sigma\zeta  \{ {(m+q)r} \}^{1/2} \|\bsbX \bsbB'   + \bsbC'\|_{\tF}
$
for all $\bsbB' $  and  $\bsbC' $   with  $r(\bsbB' )\le 2r$,   where   $\vartheta>0$ is a constant and $\zeta\ge 0$. Then, we have  \begin{align*}
\EE (\|  \hat \bsbB -\bsbB^*\|_{\tF}^2)\lesssim \,   \sigma^2   (1+\zeta^2)\frac{(m+q)r }{n}\mbox{.}
 \end{align*}
\end{theorem}

\begin{proof}
A careful examination of  the proof  of Theorem 3 shows  that   for any $a\ge 2 b > 0$,
\begin{align*}
(1-\frac{1}{a})M(\hat\bsbB-\bsbB^*, \hat\bsbC - \bsbC^*)  
\leq \ &    2aA_0 \sigma^2 r(m+q) + R  +2  P(\bsbC^{*}; {\lambda}) - 2P(\hat \bsbC; {\lambda}) \\& +  \frac{1}{b}    P_{2,H}(\hat \bsbC - \bsbC^*; {\lambda}),
\end{align*}
where    $\lambda = A\lambda^o$, $\lambda^o=\sigma(m+ \log n)^{1/2}$, $A =  (a b A_1)^{1/2}$, $A_1\ge A_0$ with $A_0$ a universal constant, and $ER\le a c \sigma^2$.

Set    $b = 1/(2\theta)$, $\theta = \vartheta/(2+\vartheta)$. Then
\begin{align*}
(1-\frac{1}{a})M(\hat\bsbB-\bsbB^*, \hat\bsbC - \bsbC^*)  
\leq \, &   2 (1-\theta)\lambda \{(1+\vartheta)   \| (\hat \bsbC - \bsbC^*)_{\mathcal J^*}\|_{2,1}  - \| (\hat \bsbC - \bsbC^*)_{{{\mathcal J}^*}^c}\|_{2,1} \}\\&+2aA_0 \sigma^2 r(m+q) + R    \\
\le  \, &  2 (1-\theta)\Big[\sigma \zeta  \{ {(m+q)r} \}^{1/2} \{M(\hat\bsbB-\bsbB^*, \hat\bsbC - \bsbC^*) \}^{1/2}  \\ & -n\| \hat \bsbB - \bsbB^*\|_{\tF}^2 \Big]+ 2aA_0 \sigma^2 r(m+q) + R\mbox{.}
\end{align*}
The conclusion follows by applying   H{\"o}lder's inequality and   setting, say,    $a = 2 + 1/\theta$, $b = 1/2\theta$
and $A\ge (a b A_0)^{1/2}$.
\end{proof}


\section{Simulations}
\label{sec:sim}


\subsection{Simulation setups}

We consider three model setups. In Models I and II, we set $n=100$, $p=12$, $m=8$, and $r^*=3$. The design matrix $\bsbX$ is generated by sampling its $n$ rows from $N(\bsb0,\bsbDelta_0)$, where $\bsbDelta_0$ is with diagonal elements 1 and off-diagonal elements 0.5. This brings in wide-range predictor correlation. The rows of the error matrix $\bsbE$ are generated as independently and identically distributed samples from $N(\bsb0,\sigma^2\bsbSig_0)$. Models I and II differ in their error structures. In Model I, we set $\bsbSig_0=\bsbI$, whereas in Model II, $\bsbSig_0$ has the same compound symmetry structure as $\bsbDelta_0$. In each simulation, $\sigma^2$ is computed to control the signal to noise ratio, defined as the ratio between the $r^*$th singular value of $\bsbX\bsbB^*$ and $\|\bsbE\|_{\textrm{F}}$. 

Model III is a high-dimensional setup with $n=100$, $p=500$, $m=50$, $r^*=3$ and $q=10$. As such, there are 25,000 unknown parameters in the coefficient matrix, posing a challenging high-dimensional problem. The design is generated as $\bsbX=\bsbX_1\bsbX_2\bsbDelta_0^{1/2}$, where $\bsbX_1\in \mathbb{R}^{n\times q}$, $\bsbX_2\in \mathbb{R}^{q\times p}$, and all entries of $\bsbX_1$ and $\bsbX_2$ are independently and identically distributed samples from $N(0, 1)$. The error structure is the same as in Model II. 


In each of the three models, $\bsbB^*$ is randomly generated as $\bsbB^* = \bsbB_1\bsbB_2^\T$ in each simulation, where $\bsbB_1 \in \mathbb{R}^{p\times r^*}$, $\bsbB_2\in \mathbb{R}^{m\times r^*}$ and all entries in $\bsbB_1$ and $\bsbB_2$ are independently and identically distributed samples from $N(0,1)$. Outliers are then added by setting the first $n\times O\%$ rows of $\bsbC^*$ to be  nonzero, where $O\% \in \{5\%,10\%,15\%\}$. 
Concretely, the $j$th entry in any outlier row of $\bsbC^*$ is   $\alpha$ times the standard deviation of the $j$th column of $\bsbX\bsbB^*$, where $1\leq j\leq m$ and $\alpha=2,4$. To make the problem even more challenging, we modify all entries of the first two rows of the design  to $10$. This yields some outliers with high leverage values. Finally, the response $\bsbY$ is generated as $\bsbY = \bsbX\bsbB^* + \bsbC^* + \bsbE$. Overall,  the  signal  is contaminated by both random errors and gross outliers. Under each setting, the entire data generation process described above is replicated 200 times.




\subsection{Methods and evaluation metrics}

We compare the proposed robust reduced-rank regression with several robust regression approaches and rank reduction methods. There exist many robust multivariate regression methods in the traditional large-$n$ setting. We mainly consider the MM-estimator by \citet{Tatsuoka2000}, using its implementation provided by the R package \texttt{FRB} and the default settings therein. Other robust estimators including the S-estimator \citep{VanAelst2005} and the GS-estimator \citep{Roelant2009} were also examined; we omit their results here, as they were similar to or slightly worse than those of the MM-estimator. None of  these classical methods is  applicable in high dimensions, and so   they were only used on the datasets generated according to  Models I and II.


For reduced-rank methods, we consider the plain reduced-rank regression \citep{bunea2011} and the reduced-rank ridge regression \citep{mukh2011,She2013}, both tuned by 10-fold cross validation. The latter method combines rank reduction and shrinkage estimation, which can potentially improve the predictive performance of the former when the predictors exhibit strong correlation.

We also consider a three-step fitting-detection-refitting procedure. Specifically, the first step is to fit a plain reduced-rank regression using all data; in the second step, the value of the residual sum of squares is computed for each of the $n$ observation rows, and exactly $n\times O\%$ observations with the largest residual sum of squares are labeled as outliers and discarded; at the third step, the plain reduced-rank regression is refitted with the rest of the observations. This method can be regarded as a naive oracle procedure, as it relies on the knowledge of the true number of outliers.

As for the proposed robust reduced-rank regression, we used the $\ell_0$ penalized form and  the predictive information criterion    for tuning. Our method allows the incorporation of the error structure through setting the weighting matrix $\bsbGamma$; see Equation \eqref{eq:r4Gamma} of the paper. To investigate the impact of weighting, we considered both $\bsbGamma=\bsbI$ and $\bsbGamma=\hat{\bsbSig}^{-1}$ in the setting of Model II, where $\hat{\bsbSig}$ is a robust estimate of $\bsbSig=\sigma^2\bsbSig_0$ from MM-estimation. Since it is in general difficult to estimate $\bsbSig$ in high dimensional settings, for the data generated in Model III we just set $\bsbGamma=\bsbI$. For each rank value $r=1,\ldots,\min(n,q)$, we compute the solutions over a grid of 100 $\lambda$ values equally spaced on the log scale, corresponding to a proper interval of the proportion of outliers given by $[v_L,v_U]$. We take $v_L=0$ and $v_U\approx0.4$, as in practice the proportion of outliers is usually under 40\%. All the methods are implemented in a user-friendly R package.





To characterize estimation accuracy robustly, we report the 10\% trimmed mean of the mean squared error from all runs,
$$
\mbox{Err}(\hat{\bsbB})=\|\bsbX\bsbB^*-\bsbX\hat{\bsbB}\|_{\tF}^2/(mn)\mbox{.}
$$
In Model II, we additionally report the 10\% trimmed mean of the weighted mean squared errors from all runs, defined as
$$
\mbox{Err}(\hat{\bsbB};\bsbSig)=\tr\{(\bsbX\bsbB^*-\bsbX\hat{\bsbB})\bsbSig^{-1}(\bsbX\bsbB^*-\bsbX\hat{\bsbB})^\T\}/(mn),
$$
where $\bsbSig=\sigma^2\bsbSig_0$ is the true error covariance matrix. Similarly, the prediction error is defined as
$$
\mbox{Err}(\hat\bsbB,\hat\bsbC) =\|\bsbX\bsbB^* + \bsbC^*-\bsbX\hat{\bsbB}-\hat{\bsbC}\|_{\tF}^2/(mn)\mbox{.}
$$
While the robust reduced-rank regression explicitly   estimates $\bsbC^*$, this is not the case for the other approaches. In the plain reduced-rank regression and the reduced-rank ridge regression, $\hat{\bsbC}$ is set as a zero matrix, while in the MM estimation and the three-step procedure, the rows in $\hat{\bsbC}$ corresponding to the identified outliers are filled with model residuals in $\bsbY - \bsbX\hat{\bsbB}$. The leverage points, if exists, are removed from $\bsbX$ in the above calculations.

To evaluate the rank selection performance, we report the average of rank estimates from all runs. To examine the outlier detection performance, we report the average masking rate, i.e., the fraction of undetected outliers, the average swamping rate, i.e., the fraction of good points labeled as outliers, and the frequency of correct joint outlier detection, i.e., the fraction of simulations with no masking and no swamping.


\subsection{Simulation results}


Tables \ref{table1}--\ref{table3} summarize the simulation results of Models I--III, respectively, for $\alpha=2$ and signal to noise ratio 0.75. We omit the results in other settings since they deliver similar messages.

In Models I and II, the MM-estimates achieved better predictive performance than both reduced-rank regression and reduced-rank ridge regression. This demonstrates that when severe outliers are present, it is pivotal to perform robust estimation. 
Even in these low-dimensional   settings, the proposed robust reduced-rank regression  outperforms all other methods, and perfectly detects all   outliers jointly.   MM-estimation can also  achieve  pretty low masking rates, but this comes at the cost of increasing   false positives, which translates to efficiency loss. In particular,   when the errors  become correlated, our robust reduced-rank regression still showed impressive performance in both prediction and outlier detection. Additionally, the inverse covariance weighting did show some improvements over the identity weighting, but the gain was small.



Both reduced-rank regression and reduced-rank ridge regression tended to overestimated the rank in the presence of highly  leveraged outliers. This complies with the theoretical results, cf. Remark \ref{remTh5} following Theorem \ref{th_oracle-l1}. In contrast,   robust reduced-rank regression achieved nearly  perfect rank selection in all the experiments. The three-step procedure relies on the      accuracy of the estimated model residuals, and   often fails in the presence of leverage points. In practice, making a judgement of  the   number of  outliers is critical. One merit of the proposed method   is that  the theoretically justified predictive information criterion  can   choose suitable parameters regardless of the size of $n$, $m$, or $p$, leading to an automatic identification of the right amount of outlyingness from a predictive learning perspective.  






Similar conclusions can be drawn from the comparison in the high-dimensional model. Indeed, according to Table \ref{table3},  the robust reduced-rank regression showed comparable or better performance than the other methods in almost all categories.

\begin{table}[htp]
\caption{\label{table1} Simulation  results of Model I with $\alpha=2$ and signal to noise ratio 0.75. The errors are reported with their standard errors in parentheses}
\centering
\begin{tabular}{lrrrrrr}
           & Err($\hat{\bsbB}$) & Err($\hat{\bsbB}$, $\hat{\bsbC}$) &      Rank  &       Mask &      Swamp &  Detection \\

           &                                                    \multicolumn{ 6}{c}{5\%} \\

        MM &  0$\cdot$4 (0$\cdot$2) &  4$\cdot$2 (1$\cdot$7) &        8$\cdot$0 &        0\% &      3$\cdot$7\% &        0\% \\

       RRR &  2$\cdot$9 (3$\cdot$7) &  6$\cdot$1 (4$\cdot$4) &        3$\cdot$6 &      100\% &        0\% &        0\% \\

       RRS &  1$\cdot$8 (0$\cdot$8) &  4$\cdot$7 (1$\cdot$7) &        4$\cdot$0 &      100\% &        0\% &        0\% \\

       RRO &  0$\cdot$3 (0$\cdot$3) &    1$\cdot$2 (1) &        3$\cdot$1 &     18$\cdot$1\% &        1\% &     28$\cdot$5\% \\

$\mbox{R}^4$ &  0$\cdot$2 (0$\cdot$1) &  0$\cdot$3 (0$\cdot$1) &        3$\cdot$0 &        0\% &        0\% &      100\% \\


           &                                                   \multicolumn{ 6}{c}{10\%} \\

        MM &  0$\cdot$4 (0$\cdot$2) &   12$\cdot$3 (6) &        8$\cdot$0 &        0\% &      2$\cdot$6\% &      1$\cdot$5\% \\

       RRR &    5$\cdot$4 (5) & 15$\cdot$9 (8$\cdot$5) &        3$\cdot$5 &      100\% &        0\% &        0\% \\

       RRS &  3$\cdot$5 (2$\cdot$4) & 14$\cdot$3 (9$\cdot$7) &        4$\cdot$1 &      100\% &        0\% &        0\% \\

       RRO &  0$\cdot$3 (0$\cdot$2) &    2 (1$\cdot$3) &        3$\cdot$0 &     13$\cdot$3\% &      1$\cdot$5\% &     20$\cdot$5\% \\

$\mbox{R}^4$ &  0$\cdot$2 (0$\cdot$1) &  0$\cdot$4 (0$\cdot$2) &        3$\cdot$0 &        0\% &        0\% &      100\% \\


           &                                                   \multicolumn{ 6}{c}{15\%} \\

        MM &  0$\cdot$5 (0$\cdot$4) & 17$\cdot$8 (6$\cdot$6) &        8$\cdot$0 &      0$\cdot$1\% &      1$\cdot$4\% &       24\% \\

       RRR &  4$\cdot$4 (2$\cdot$1) & 17$\cdot$9 (5$\cdot$5) &        3$\cdot$8 &      100\% &        0\% &        0\% \\

       RRS &    4 (2$\cdot$5) & 18$\cdot$4 (6$\cdot$1) &        3$\cdot$9 &      100\% &        0\% &        0\% \\

       RRO &  0$\cdot$5 (0$\cdot$3) &  2$\cdot$3 (1$\cdot$5) &        3$\cdot$0 &      8$\cdot$9\% &      1$\cdot$6\% &     27$\cdot$5\% \\

$\mbox{R}^4$ &  0$\cdot$3 (0$\cdot$2) &  0$\cdot$8 (0$\cdot$5) &        2$\cdot$9 &        0\% &        0\% &      100\% \\


\end{tabular}
\end{table}

\begin{table}[htp]
\caption{\label{table2} Simulation  results of Model II with $\alpha=2$ and signal to noise ratio 0.75.  The layout of the table is similar to that of Table \ref{table1}}
\centering
\begin{tabular}{lrrrrrrr}

           & Err($\hat{\bsbB}$) & Err($\hat{\bsbB}$; $\Sigma$) & Err($\hat{\bsbB}$, $\hat{\bsbC}$) &      Rank  &       Mask &      Swamp &  Detection \\

           &                                                                 \multicolumn{ 7}{c}{5\%} \\

        MM &  0$\cdot$4 (0$\cdot$3) &  0$\cdot$4 (0$\cdot$3) &  6$\cdot$9 (2$\cdot$9) &        8$\cdot$0 &        0\% &      3$\cdot$3\% &        0\% \\

       RRR &  2$\cdot$6 (2$\cdot$4) &  4$\cdot$6 (4$\cdot$3) &  9$\cdot$8 (6$\cdot$2) &        4$\cdot$0 &      100\% &        0\% &        0\% \\

       RRS &  1$\cdot$9 (1$\cdot$4) &  3$\cdot$3 (2$\cdot$5) &  8$\cdot$5 (4$\cdot$4) &        4$\cdot$3 &      100\% &        0\% &        0\% \\

       RRO &  0$\cdot$4 (0$\cdot$3) &  0$\cdot$5 (0$\cdot$3) &  2$\cdot$7 (1$\cdot$8) &        3$\cdot$0 &     25$\cdot$7\% &      1$\cdot$4\% &       17\% \\

       $\mbox{R}^4$ &  0$\cdot$2 (0$\cdot$2) &  0$\cdot$2 (0$\cdot$2) &  0$\cdot$3 (0$\cdot$2) &        3$\cdot$0 &        0\% &      0$\cdot$2\% &       84\% \\


       $\mbox{R}^4_w$ &  0$\cdot$2 (0$\cdot$1) &  0$\cdot$2 (0$\cdot$2) &  0$\cdot$3 (0$\cdot$2) &        3$\cdot$0 &        0\% &        0\% &      100\% \\

           &                                                                \multicolumn{ 7}{c}{10\%} \\

        MM &  0$\cdot$5 (0$\cdot$3) &  0$\cdot$5 (0$\cdot$4) & 21$\cdot$2 (9$\cdot$7) &        8$\cdot$0 &        0\% &      1$\cdot$9\% &     12$\cdot$5\% \\

       RRR &  3$\cdot$6 (1$\cdot$1) &  6$\cdot$5 (2$\cdot$3) & 21$\cdot$7 (9$\cdot$1) &        4$\cdot$1 &      100\% &        0\% &        0\% \\

       RRS &    4 (1$\cdot$8) &  7$\cdot$4 (3$\cdot$7) & 24$\cdot$6 (10$\cdot$6) &        4$\cdot$0 &      100\% &        0\% &        0\% \\

       RRO &  0$\cdot$4 (0$\cdot$2) &  0$\cdot$6 (0$\cdot$3) &  4$\cdot$3 (2$\cdot$1) &        3$\cdot$0 &     16$\cdot$4\% &      1$\cdot$8\% &      4$\cdot$5\% \\

        $\mbox{R}^4$ &  0$\cdot$3 (0$\cdot$2) &  0$\cdot$4 (0$\cdot$3) &  0$\cdot$7 (0$\cdot$6) &        3$\cdot$0 &        0\% &        0\% &     99$\cdot$5\% \\


       $\mbox{R}_w^4$ &  0$\cdot$2 (0$\cdot$1) &  0$\cdot$3 (0$\cdot$2) &  0$\cdot$6 (0$\cdot$4) &        3$\cdot$0 &        0\% &        0\% &      100\% \\

           &                                                                \multicolumn{ 7}{c}{15\%} \\

        MM &  0$\cdot$4 (0$\cdot$2) &  0$\cdot$4 (0$\cdot$2) & 31$\cdot$3 (12$\cdot$4) &        8$\cdot$0 &        0\% &      1$\cdot$1\% &     46$\cdot$5\% \\

       RRR &  4$\cdot$5 (2$\cdot$7) &  7$\cdot$9 (5$\cdot$2) & 33$\cdot$4 (13$\cdot$4) &        4$\cdot$3 &      100\% &        0\% &        0\% \\

       RRS &  4$\cdot$8 (3$\cdot$4) &  8$\cdot$7 (6$\cdot$8) & 36$\cdot$5 (16$\cdot$1) &        4$\cdot$0 &      100\% &        0\% &        0\% \\

       RRO &  0$\cdot$4 (0$\cdot$2) &  0$\cdot$6 (0$\cdot$2) &  3$\cdot$3 (1$\cdot$4) &        3$\cdot$0 &      9$\cdot$4\% &      1$\cdot$7\% &       10\% \\

        $\mbox{R}^4$ &  0$\cdot$2 (0$\cdot$2) &  0$\cdot$3 (0$\cdot$2) &  0$\cdot$6 (0$\cdot$3) &        3$\cdot$0 &      0$\cdot$3\% &        0\% &     95$\cdot$5\% \\


       $\mbox{R}_w^4$ &  0$\cdot$2 (0$\cdot$1) &  0$\cdot$2 (0$\cdot$1) &  0$\cdot$5 (0$\cdot$2) &        3$\cdot$0 &        0\% &        0\% &      100\% \\

\end{tabular}
\end{table}

\begin{table}[htp]
\caption{\label{table3} Simulation  results of Model III with $\alpha=2$ and signal to noise ratio 0.75. The values of actual Err($\hat{\bsbB}$) and Err($\hat{\bsbB}$, $\hat{\bsbC}$) are divided by 100 for better presentation. The layout of the table is similar to that of Table \ref{table1}}
\centering
\begin{tabular}{lrrrrrr}

           & Err($\hat{\bsbB}$) & Err($\hat{\bsbB}$, $\hat{\bsbC}$) &      Rank  &       Mask &      Swamp &  Detection \\

           &                                                    \multicolumn{ 6}{c}{5\%} \\

       RRR &  2$\cdot$5 (0$\cdot$9) & 15$\cdot$5 (6$\cdot$3) &        4$\cdot$0 &      100\% &        0\% &        0\% \\

       RRS &  2$\cdot$4 (0$\cdot$9) & 15$\cdot$6 (6$\cdot$3) &        4$\cdot$0 &      100\% &        0\% &        0\% \\

       RRO &    1 (0$\cdot$6) &  3$\cdot$9 (3$\cdot$9) &        3$\cdot$0 &     11$\cdot$3\% &      0$\cdot$6\% &     67$\cdot$5\% \\

$\mbox{R}^4$ &  0$\cdot$9 (0$\cdot$5) &  1$\cdot$6 (0$\cdot$9) &        3$\cdot$0 &      1$\cdot$6\% &        0\% &       96\% \\


           &                                                   \multicolumn{ 6}{c}{10\%} \\

       RRR &  5$\cdot$4 (2$\cdot$3) &  47$\cdot$5 (18) &        4$\cdot$0 &      100\% &        0\% &        0\% \\

       RRS &  5$\cdot$1 (2$\cdot$1) &  47$\cdot$8 (18) &        4$\cdot$0 &      100\% &        0\% &        0\% \\

       RRO &  0$\cdot$8 (0$\cdot$4) &  5$\cdot$1 (4$\cdot$6) &        3$\cdot$0 &      4$\cdot$9\% &      0$\cdot$5\% &     68$\cdot$5\% \\

$\mbox{R}^4$ &  0$\cdot$7 (0$\cdot$3) &  2$\cdot$2 (0$\cdot$9) &        3$\cdot$0 &        0\% &        0\% &      100\% \\


           &                                                   \multicolumn{ 6}{c}{15\%} \\

       RRR &  8$\cdot$7 (4$\cdot$2) &  77 (39$\cdot$9) &        4$\cdot$0 &      100\% &        0\% &        0\% \\

       RRS &    8 (3$\cdot$6) &  77$\cdot$4 (40) &        4$\cdot$0 &      100\% &        0\% &        0\% \\

       RRO &  1$\cdot$4 (0$\cdot$8) & 11$\cdot$9 (8$\cdot$5) &        3$\cdot$0 &      9$\cdot$7\% &      1$\cdot$7\% &       24\% \\

$\mbox{R}^4$ &  0$\cdot$8 (0$\cdot$3) &  3$\cdot$1 (1$\cdot$1) &        3$\cdot$2 &      3$\cdot$2\% &        0\% &     75$\cdot$5\% \\


\end{tabular}
\end{table}

\clearpage
\bibliographystyle{Chicago}

\bibliography{r4}{}


\end{document}